\documentclass[11pt]{amsart}
\usepackage{amsmath}
\usepackage{amsmath,amssymb,enumitem,ulem,multicol}
\usepackage{mathabx}
\usepackage{amssymb,tikz}
\usepackage{amsthm}
\usepackage{upgreek} 
\usepackage{bigints}
\usepackage{parskip}
\usepackage{setspace}
\usepackage{mathrsfs}
\usepackage{mathtools}
\usepackage{xcolor}
\usepackage{xcolor}
\usepackage{float}
\usepackage[utf8]{inputenc}
\DeclareUnicodeCharacter{2212}{-}
\usepackage[english]{babel}
\usepackage{graphicx}
\usepackage{subcaption}
\theoremstyle{remark}

\def\R{\mathbb{R}}

\def\cN{\mathcal{N}}
\def\cA{\mathcal{A}}
\def\cK{\mathcal{K}}
\def\cT{\mathcal{T}}
\def\cE{\mathcal{E}}

\def\p{\partial}

\def\[{\partial}
\def\O{\Omega}
\def\ssT{{\scriptscriptstyle T}}
\def\HT{{H^2(\O,\cT_h)}}
\def\mean#1{\left\{\hskip -5pt\left\{#1\right\}\hskip -5pt\right\}}
\def\jump#1{\left[\hskip -3.5pt\left[#1\right]\hskip -3.5pt\right]}
\def\smean#1{\{\hskip -3pt\{#1\}\hskip -3pt\}}
\def\sjump#1{[\hskip -1.5pt[#1]\hskip -1.5pt]}
\def\jumptwo{\jump{\frac{\p^2 u_h}{\p n^2}}}
\def\b#1{\boldsymbol{#1}}
\def\norm #1{{\left\vert\kern-0.25ex\left\vert\kern-0.25ex\left\vert #1 
    \right\vert\kern-0.25ex\right\vert\kern-0.25ex\right\vert}}
\usepackage{ragged2e}
\usepackage{amssymb}
\usepackage{amsmath}
\usepackage{amsthm}
\usepackage{amsbsy}
\usepackage{psfrag}
\usepackage{pstricks}
\usepackage{graphics}
\usepackage{graphicx}
\usepackage{tgtermes}
\fontfamily{qcr}
\theoremstyle{plain}
\newtheorem{theorem}{Theorem}[section]
\newtheorem{lemma}[theorem]{Lemma}

\newtheorem{example}[theorem]{Example}
\theoremstyle{remark}
\newtheorem{remark}[theorem]{Remark}

\begin{document}
\allowdisplaybreaks[4]
\numberwithin{figure}{section}
\numberwithin{table}{section}
 \numberwithin{equation}{section}
%
\title[ Adaptive Quadratic Finite Element method for Unilateral Contact Problem]
 {Adaptive quadratic finite element method for the unilateral contact problem}
 \author{Rohit Khandelwal}\thanks{}
\address{Department of Mathematics, Indian Institute of Technology Delhi - 110016}
\email{rohitkhandelwal004@gmail.com }
\author{Kamana Porwal}\thanks{The second author's work is supported  by CSIR Extramural Research Grant}
\address{Department of Mathematics, Indian Institute of Technology Delhi - 110016}
\email{kamana@maths.iitd.ac.in}
\author{Tanvi Wadhawan}\thanks{}
\address{Department of Mathematics, Indian Institute of Technology Delhi - 110016}
\email{tanviwadhawan1234@gmail.com}
\date{}

\begin{abstract}
In this paper, we present and analyze a posteriori error estimates in the energy norm of a quadratic finite element method for the frictionless unilateral
contact problem.  The reliability and the efficiency of a posteriori error estimator is discussed. The suitable decomposition of the discrete space $\b{V^h}$ and a discrete space $\b{Q^h}$, where the discrete counterpart of the contact force density is defined, play crucial role in  deriving a posteriori error estimates. Numerical results are presented exhibiting the reliability and the efficiency of the proposed error estimator.
\end{abstract}

%
%
\maketitle
\allowdisplaybreaks
\def\R{\mathbb{R}}
\def\cA{\mathcal{A}}
\def\cK{\mathcal{K}}
\def\cN{\mathcal{N}}
\def\p{\partial}
\def\O{\Omega}
\def\bbP{\mathbb{P}}
\def\cV{\mathcal{V}}
\def\cM{\mathcal{M}}
\def\cT{\mathcal{T}}
\def\cE{\mathcal{E}}
\def\bF{\mathbb{F}}
\def\bC{\mathbb{C}}
\def\bN{\mathbb{N}}
\def\ssT{{\scriptscriptstyle T}}
\def\HT{{H^2(\O,\cT_h)}}
\def\mean#1{\left\{\hskip -5pt\left\{#1\right\}\hskip -5pt\right\}}
\def\jump#1{\left[\hskip -3.5pt\left[#1\right]\hskip -3.5pt\right]}
\def\smean#1{\{\hskip -3pt\{#1\}\hskip -3pt\}}
\def\sjump#1{[\hskip -1.5pt[#1]\hskip -1.5pt]}
\def\jumptwo{\jump{\frac{\p^2 u_h}{\p n^2}}}

\section{Introduction}\label{Intro}
Numerical analysis of the non-linear problems arising from unilateral contact problems using finite element methods exhibits technical adversity both in approximating the continuous problem and numerical modeling of contact conditions on a part of the boundary. The Signorini contact model typically is a prototype model for the class of unilateral contact problems \cite{KO:1988:CPBook}. The Signorini contact problem can be recasted as an elliptic variational inequality of the first kind \cite{Glowinski:2008:VI} where the inequality constraint arises due to non linearity condition on the contact boundary. Later, the location of the free boundary (the part of the boundary where it touches the given obstacle) is not a priori known, and therefore, it forms a part of the numerical approximation. Hence, it is quite challenging both in the theory and computation to analyze finite element approximation of the Signorini problem using quadratic elements.

Adaptive finite element methods (AFEM) \cite{AO:2000:Book,Verfurth:1995:AdaptiveBook} are considered as an essential tool in boosting the precision of the numerical approximation of the non-linear problems. AFEM is mainly based on the reliable and an (locally) efficient a posteriori error estimators which are known quantities that depends on the given data and discrete solution. Subsequently, 
there has been a tremendous work on the analysis and development of finite element methods for variational
inequalities. We refer to articles \cite{Glowinski:2008:VI,KO:1988:CPBook,Belgacem:2000:NSUCP,Frenco:1977:FEVI,Francesco:1977:FEUP} for convergence  analysis of the Signorini problem using linear finite element method, where as in \cite{Belhachmi:2003:QuadSig,Hild:2002:QuadraticFEM} a priori analysis of quadratic finite element method has been derived for the contact problem. An extensive study of convergence analysis of discontinuous Galerkin (DG) methods for simplified Signorini problem has been carried out in \cite{Wang:2010:DGEV}. {{In the monograph \cite{Han:DG:2011} several  DG methods have been discussed for the Signorini problem and therein a priori error analysis have been established.  Adaptive conforming finite element method for the Signorini problem has been discussed in \cite{Krause:2015:apost_Sig,Hild:2007:CPE,HN:2005,Weiss:2004:Cf}. The articles \cite{TG:2016:VIDG1, Walloth:2019:Dg} analyze a posteriori error analysis of discontinuous Galerkin finite element methods for the Signorini problem. Note that, the solution of the Signorini problem may not be of class $H^{3}$ because of the presence of the free boundary thereby the uniform refinement does not yield the optimal convergence using quadratic finite element approximation (see \cite{Hild:2002:QuadraticFEM}) but adaptive refinement gives the optimal convergence.
In this paper, we derive the residual based a posteriori error estimates for the quadratic conforming finite element method for the unilateral contact problem}}. 
To the best of the knowledge of the authors, quadratic AFEM for the Signorini problem has not been discussed so far. 
One of the key ingredient of our analysis is the appropriate construction of the discrete counterpart of the continuous contact force density which helps in proving the main results of this article. 

\par

The outline of this article is as follows. In Section \ref{sec:Prelim}, we introduce continuous contact force density and some notations which are used in later analysis. Therein, we also present the continuous (strong and weak) formulation of the Signorini problem and discuss some preliminary results. Section \ref{sec:Discrete} is devoted to the introduction of discrete spaces  on which discrete problem and discrete contact force density are defined followed by introducing the discrete Lagrange multiplier and deriving its basic properties. Further, in Section \ref{sec4}, we introduce quasi discrete contact force which imitates the property of continuous Lagrange multiplier but computed using discrete contact force density. In Section \ref{sec:Apoestriori}, we propose and analyze a posteriori error estimator, therein the reliability and efficiency of the error estimator is discussed. Finally, numerical experiments illustrating the convergence behavior of proposed a posteriori error estimator using quadratic finite elements are depicted in Section \ref{sec:NumResults}.
\par
Let $\O\subset \mathbb R^2$ represents a bounded, polygonal elastic body with Lipschitz boundary $\partial \Omega=\Gamma$ which is partitioned into three non overlapping, relatively open parts $\Gamma_D$, $\Gamma_N$ and $\Gamma_C$ with $meas(\Gamma_D) > 0$ and $\overline \Gamma_C \subset \Gamma~ \backslash~\overline \Gamma_D$ where $meas(A)$ denotes the measure of any set $A \subset \bar{\Omega}$. Let $\b{e_1}$ and $~\b{e_2}$ denotes the standard ordered basis functions of $\mathbb{R}^2$. 

\par
\section{Basic Preliminaries and Definitions} \label{sec:Prelim}
We recall here some basic notations associated with the finite element setting which are required in the subsequent sections:
\begin{itemize}[noitemsep]
    \item  $\mathcal{T}_h$ is a family of regular triangulation of $\O$,
   
    \item $\mathcal{E}_h$ denotes set of all edges of $\mathcal{T}_h$ ,
    \item $\mathcal{E}_h^i$ denotes set of all interior edges of $\mathcal{T}_h$,
    \item  $\mathcal{E}_h^b$ denotes set of all boundary edges of $\mathcal{T}_h$,
  
    \item $\mathcal{E}_h^N$ denotes set of all boundary edges lying on  $\Gamma_N$, 
    \item $\mathcal{E}_h^C$ denotes set of all boundary edges lying on  $\Gamma_C$,
   
     \item $\mathcal{V}_h$ denotes set of all the vertices of $\mathcal{T}_h$,
     \item $\mathcal{M}_h$ denotes set of all the midpoints of edges of $\mathcal{T}_h$,
    \item $\mathcal{V}_{e}$ denotes set of vertices lying on edge $e$,
    \item $\mathcal{M}_{e}$ refers to the midpoint of the edge $e$,
    \item $\mathcal{V}_h ^C$ denotes the set of vertices of $\mathcal{T}_h$ lying on $\overline{\Gamma_C}$,
    \item $\mathcal{V}_h^D$ denotes the set of vertices of $\mathcal{T}_h$ lying on $\overline{ \Gamma_D}$,
     \item $\mathcal{M}_h^C$ denotes the set of midpoint of the edges lying on~ $\Gamma_C$,
     \item $\mathcal{M}_h^D$ denotes the set of midpoint of the edges lying on~ $\Gamma_D$,
     \item $\mathcal{V}_h^o$ refers to $\mathcal{V}_h$ $\backslash$ $\mathcal{V}_h^D$, 
    \item $\mathcal{M}_h^o$ refers to $\mathcal{M}_h$ $\backslash$ $\mathcal{M}_h^D$, 
     \item $T$ is an element of $\mathcal{T}_h$,
   \item  $h_T$ is the diameter of $T$ where $T\in \mathcal{T}_h$,
    \item $h$  refers to maximum of the set $\{h_T : T \in \mathcal{T}_h\}$,
   \item  $h_e$ is the length of an edge $e$,
   \item $\omega_p$ refers to the set of all elements sharing the node $p$,
   \item $h_p$ refers to maximum of the set $\{h_T : T \in \omega_p$\},
    \item $\gamma_{p,N} := \partial \omega_p \cap \Gamma_N$,
     \item $\gamma_{p,C} := \partial \omega_p \cap \Gamma_C$,
      \item $\gamma_{p,I}$ refers to all interior edges in $\omega_p$,
      \item $h_s$ refers to maximum of the set $\{h_e : e \in \gamma_p$ where $\gamma_p=\gamma_{p,I},\gamma_{p,N}$ or $\gamma_{p,C}$\},
     \item $P_{k}(T)$ denotes the space of polynomials of degree $\leq k$ defined on $T$ where $0 \leq k \in \mathbb{Z}$,
\item $|S|$ denotes the cardinality of the set $S$.
   \end{itemize}
 \par
Next, we define the following differential operators and preliminary definitions for the further use:
  \begin{itemize}[noitemsep]
 \item For any Banach space $\b{X}$, let $\b{X^*}$ denotes the dual space of $(\b{X},\|\cdot\|_{\b{X}})$ with the dual norm $\|\cdot\|_*$ defined by 
\begin{align*}
    \|\b{L}\|_{\b{*}}~:= \underset{\b{v}\in \b{X},~ \b{v} \neq \b{0}}{sup}~\dfrac{~L(\b{v})~}{\|{\b{v}}\|_{\b{X}}} ~~~\forall~ \b{L} \in \b{X^*},
\end{align*}
 \item $\nabla\b{v}$ is a $2\times 2$ gradient matrix of a vector $\b{v} \in \mathbb{R}^2$,
 \item For any matrix $\b{M}= (m_{ij}) \in \mathbb{R}^{2\times 2}$, the divergence of $\b{M}$ is defined as
 \begin{align*}
  div (\b{M}) := \sum_{j=1}^{2} \frac{\partial}{\partial x_j}(m_{ij}), ~i=1,2.   
 \end{align*}
 \item $\b{\epsilon}(\b{v})$ is the linearized strain tensor defined by $\frac{1}{2} (\nabla\b{v}+\nabla\b{v}^T)$,
 {
 \item $A$ is the fourth-order elasticity tensor of the material,
 \item $\b{\sigma}(\b{v})$ is the linearized stress tensor defined by $A\b{\epsilon}(\b{v})$},
 \item $H^{m}(\O)$ denotes the usual Sobolev space \cite{BScott:2008:FEM} of square integrable functions whose weak derivative upto order $m$ is also square integrable  with the corresponding norm $\|\cdot\|_{H^m(\Omega)}$  and seminorm $|\cdot|_{H^m(\Omega)}$,
 \item  For a non integer positive number $s = m +k$, where $m$ is an integer and $0<k<1$, the fractional ordered subspace $H^{s}(\O)$ is defined as 
 \begin{align*}
  H^{s}(\O) = \{\b{v}\in H^{m}(\O) : \dfrac{{\lvert{v}(x)−{v}(y)\rvert}}{{\lvert x−y \rvert}^{k+1}} \in L^p(\Omega \times \Omega) \},
 \end{align*}
 \item For any vector $\b{v}=(v_1,v_2) \in \b{H^{m}(\O)}$=$[H^m(\O)]^2$, we define the  product norm on the domain as $\|\b{v}\|_{\b{H^m(\O)}}=\bigg(\sum_{i=1}^2\|{v_i}\|^2_{{H^m(\O)}}\bigg)^{1/2} $ and seminorm  $|\b{v}|_{\b{H^m(\O)}}=\bigg(\sum_{i=1}^2|{v_i}|^2_{{H^m(\O)}}\bigg)^{1/2}, $
 \item $\b{\langle \cdot, \cdot \rangle_{-1,1}}$ denotes the duality pairing between $\b{H^1}(\O)$ and $\b{H^{-1}(\O)},$
 \item For any $\Omega^{'} \subset \Omega$, $\langle \cdot, \cdot \rangle_{\b{-1,1},\Omega^{'}}$ denotes the duality pairing between $\b{H^1}(\O^{'})$ and $\b{H^{-1}(\O^{'})},$
 \item $\langle \cdot, \cdot \rangle_{-\frac{1}{2},\frac{1}{2}, \Gamma_C}$ denotes the duality pairing between $H^{\frac{1}{2}}(\Gamma_C)$ and $H^{-\frac{1}{2}}(\Gamma_C)$.
\item For any $v \in H^1(\O)$, we denote $v^{+}=max\{v, 0\}$ to be the positive part of the function.
  \end{itemize}

\par 
{
Throughout this article, we assume that $C$ is a positive generic constant independent of mesh parameter $h$. Further, the notation $x\lesssim y$ denotes that there is a generic constant $C$ such that $x\leq Cy$.
\par
Next, we define the broken Sobolev space $[H^1(\Omega,\mathcal{T}_h)]^2$, with the aim of defining the jump and averages of discontinuous functions efficiently as
\begin{align*}
    [H^1(\Omega,\mathcal{T}_h)]^2 := \{ \b{v}\in [L^2(\Omega)]^2 : \b{v}|_T\in [H^1(T)]^2~ \forall~ T\in \cT_h\}.
\end{align*}
\par
Let $e \in \cE_h^i$ be an interior edge and let $T^{+}$ and $T^{−}$ be the neighbouring elements s.t. $e \in \partial{T}^{+} \cup \partial T^{−}$ and let $\b{n}^{\pm}$ is the unit outward normal vector on $e$ pointing from $T^{+}$ to $T^{-}$ s.t. $\b{n^{-}}= - \b{n^{+}}.$ For a vector valued function $\b{v}\in [H^1(\Omega,\mathcal{T}_h)]^2$ and a matrix valued function $\b{\Phi}\in [H^1(\Omega,\mathcal{T}_h)]^{2\times 2}$, averages $\smean{\cdot}$ and jumps $\sjump{\cdot}$ across the edge $e$ are defined as follows:     
\begin{align*}
&\smean{\b{v}}= \frac{1}{2} (\b{v^+}+\b{v^-})~~\text{and}~~\sjump{\b{v}}= \b{v^+}\otimes \b{n^+}+\b{v^-}\otimes\b{n^-},\\
&\smean{\b{\Phi}}= \frac{1}{2} (\b{\Phi^+}+\b{\Phi^-})~ \text{and}~~\sjump{\b{\Phi}}=\b{\Phi^+}\b{n^+}+ \b{\Phi^-}\b{n^-},
\end{align*}
where
    $\b{v^{\pm}}=\b{v}|_{{T}^\pm},~\b{\Phi^{\pm}}=\b{\Phi}|_{{T}^\pm}.$
\par
For any $e \in \cE_h^b$, it is clear that there is a triangle $T \in \mathcal{T}_h$ such that $e\in \partial T \cap \partial\O$. Let $\b{n_e}$ be the unit normal of $e$ that points outside $T$. Then, the averages $\smean{\cdot}$ and jumps $\sjump{\cdot}$ of vector valued function $\b{v}\in [H^1(\Omega,\mathcal{T}_h)]^2$ and a matrix valued function $\b{\Phi}\in [H^1(\Omega,\mathcal{T}_h)]^{2\times 2}$ are defined as follows:
\begin{align*}
&\smean{\b{v}}=\b{v},~~\text{and}~~\sjump{\b{v}}=\b{v}\otimes \b{n_e},
\\
&\smean{\b{\Phi}}= \b{\Phi},~~\text{and}~~\sjump{\b{\Phi}}=\b{\Phi}\b{n_e}.
\end{align*}
In the above definitions $\b{v}\otimes \b{n}$ is a $2\times 2$ matrix with $v_in_j$ as its $(i,j)^{th}$ entry. }
\par
For any displacement field $\b{v}$, we adopt the notation $v_m=\b{v}\cdot \b{m}$ and $\b{v}_\tau=\b{v}-v_n\b{m}$, respectively, as its normal and tangential component on the boundary where $\b{m}$ is the outward unit normal vector to $\Gamma$. Similarly, for a tensor-valued function $\b{\Phi}$, the normal and tangential components are defined as ${\Phi}_m=\b{\b{\Phi} m}\cdot \b{m}$ and $\b{\b{\Phi}}_\tau=\b{\b{\Phi}}\b{m}-\b{\Phi}_m \b{m}$, respectively. Further, we have the following decomposition formula
\begin{align*}
  (\b{\Phi}\b{m})\cdot\b{v}= \Phi_mv_m+ \b{\Phi}_\tau\cdot\b{v_\tau}.
\end{align*}
In the further analysis, for any tensor-valued function $\b{\Phi}$ the term $\hat{\b{\Phi}}(\b{v})$ denotes the boundary contact stresses in the direction of the normal at the potential contact boundary and is equal to $\b{\Phi(\b{v})n}$ where $\b{n}$ is the outward unit normal on $\Gamma_C$. In this article, we assume that the outward unit normal vector $\b{n}$ to $\Gamma_C$ is constant and for the sake of simplicity, we define $\b{n}=\b{e_1}$.  
\par
In this article, we assume that our elastic body is homogeneous and  isotropic, as a result
\begin{align}\label{1.1}
 \sigma(\b{v})=A\b{\epsilon}(\b{v}):=\chi tr(\b{\epsilon}(\b{v}))I+2\mu\b{\epsilon}(\b{v}).   
\end{align}
where, $\chi>0$ and $\mu>0$ denote the Lam$\acute{e}$'s coefficients.
In order to define the continuous problem, we define the space $\b{V}$ of admissible displacements as
\begin{align*}
    \b{V}=\{ \b{v} \in [H^1(\Omega)]^2 : \b{v}=\b{0}~ \text{on}~ \Gamma_D\},
\end{align*}
and a non empty, closed and convex subset of $\b{V}$ is defined as 
\begin{align*}
    \b{\cK}=\{ \b{v} \in \b{V} : v_n=v_1\leq 0~ \text{a.e on}~ \Gamma_C\}.
\end{align*}
\par
\noindent
Given $\b{f} \in [L^2(\Omega)]^2$, $\b{g} \in [L^2(\Gamma_N)]^2$, the weak formulation of unilateral contact problem is to find $\b{u} \in \b{\cK}$ such that 
\begin{align}\label{1.2}
    a(\b{u, v - u}) \geq  L (\b{v - u})~~\forall ~~\b{v} \in \b{\cK},
    \end{align}
where, the bilinear form $a(\cdot, \cdot)$ and the linear functional $ L(\cdot)$ are defined by
    \begin{align*}
    a(\b{w, v})&=\int_\Omega \b{\sigma} (\b{w})\colon \b{\epsilon}(\b{v})~dx,\\
    L(\b{v})&= \int_\Omega \b{f} \cdot \b{v}~dx + \int_ {\Gamma_N} \b{g} \cdot \b{v}~ds~~\forall~ \b{w},\b{v}~\in~\b{V}.
    \end{align*}
The strong form associated to the variational inequality of the first kind \eqref{1.2} is to find the displacement vector $\b{u}: \Omega \rightarrow \mathbb{R}^2$ such that the following holds:
\begin{align*}
    \b{-div} ~~\b{\sigma}(\b{u}) &= \b{f} ~~~~\textit{in}~\Omega,\\
    \b{u} &= \b{0} ~~~~\textit{on}~\Gamma_D,\\
    \b{\sigma}(\b{u})\b{n} &= \b{g} ~~~~\textit{on}~\Gamma_N,\\
    u_n&\leq 0,~~\sigma_n(\b{u})\leq 0,~~{u_n\sigma_n(\b{u})=0}~~\text{and}~~\b{\sigma_\tau(\b{u})}=0~ \text{on}~\Gamma_C.
    \end{align*}
The existence and uniqueness of the solution of problem $\eqref{1.2}$ is well known from the theory of variational inequalities \cite{Glowinski:2008:VI}.

\par
Next, we define the continuous contact force density $\b{\lambda} \in \b{V^*}$ as
\begin{align} \label{CFD}
\b{\langle \b{\lambda}, \b{v} \rangle_{-1,1}} = L(\b{v}) − a(\b{u},\b{v})~\forall~\b{v}\in \b{V},
\end{align}
In the next lemma, we collect some important properties corresponding to continuous contact force density $\b{\lambda}$.
\begin{lemma} \label{CFD1}
The following holds 
\begin{align}
\b{\langle \b{\lambda}, \b{v}-\b{u}\rangle_{-1,1}} &\leq 0\quad \forall~ \b{v} \in \b{\cK}. \label{eq:SCT} \\ 
\b{\langle \b{\lambda} , \b{\phi} \rangle_{-1,1}} &\geq 0 \quad \forall~ \b{0} \leq \b{\phi} \in \b{V} . \label{eq:POOI}
\end{align}
\end{lemma}
\begin{proof}
The relation $(\ref{eq:SCT})$ can be realized directly from \eqref{1.2} and \eqref{CFD}.

To prove (\ref{eq:POOI}), let $\b{0} \leq \b{\phi} \in \b{V}$ and substitute $\b{v}= \b{u - \phi} \in \b{\cK}$ in the variational inequality \eqref{1.2} to get $\b{\langle \b{\lambda} , \b{\phi} \rangle_{-1,1}} \geq  0$.
\end{proof}
\par
{In order to realize another representation to continuous contact force density $\b{\lambda}$, we further define an intermediate space $\b{V_0}$ as
	\begin{align*}
	\b{V_0}:= \{ \b{v}=(v_1,v_2) \in \b{V},~v_1 = 0~\text{on}~ \Gamma_C\}.
	\end{align*}
	Since $\b{v}= \b{u} \pm \tilde{\b{v}} \in \b{\cK}~\forall~\tilde{\b{v}} ~\in~\b{V_0}$, therefore inequality \eqref{1.2} reduces to
	\begin{align}\label{2.6}
	a(\b{u},\b{\tilde v})~=~L(\tilde{\b{v}})~~~\forall~\tilde{\b{v}}\in \b{V_0}.
	\end{align}
	\begin{remark}
	For each $\b{v} \in \b{V}$, we have $\b{v}= (v_1,v_2):= \b{w_1}+\b{w_2}$ where $\b{w_1}=(v_1,0)$ and $\b{w_2}=(0,v_2)$. In view of equation \eqref{CFD}, we can rewrite $\langle \b{\lambda,\b{v} \rangle_{-1,1}} = \langle {\lambda}_1, {v_1} \rangle_{-1,1} +\langle {\lambda_2},{v_2} \rangle_{-1,1} $ where
	\begin{align*}
	\langle {\lambda}_1,{v_1} \rangle_{-1,1} &= L(\b{w_1}) - a (\b{u}, \b{w_1}), \\
	\langle {\lambda}_2,{v_2} \rangle_{-1,1} &= L(\b{w_2}) - a (\b{u}, \b{w_2}). 
	\end{align*}
	As $\b{w_2} \in \b{V_0}$,  using equation (\ref{2.6}), we obtain
	\begin{align*}
	\langle {\lambda}_2,{v_2} \rangle_{-1,1} &= 0 \quad \forall~ \b{v}=(v_1,v_2) \in \b{V}.
	\end{align*}
	\end{remark}
	\noindent
	Thus, we have the representation
	\begin{align} \label{3.7}
	\b{\langle \b{\lambda}, \b{v} \rangle_{-1,1}} =  \langle {\lambda}_1, {v_1} \rangle_{-1,1} \quad\forall~ \b{v}=(v_1,v_2) \in \b{V}.
	\end{align} }
An application of Green's theorem \cite{Han:DG:2011} yields 
\begin{align*}
    \langle \lambda_1, v_1 \rangle_{-1,1} &= -  \langle \hat{\b{\sigma}}_1(\b{u}), v_1 \rangle_{-\frac{1}{2},\frac{1}{2},\Gamma_C}, \\
    \langle \lambda_2, v_2 \rangle_{-1,1} &= -  \langle \hat{\b{\sigma}}_2(\b{u}), v_2 \rangle_{-\frac{1}{2},\frac{1}{2},\Gamma_C} =0.
\end{align*}

\par
In this article, we consider the approximation of the problem $\eqref{1.2}$ by quadratic finite element method. To this end, we define the finite element space $\b{V^h} \subset \b{V}$ as the space of continuous piece wise quadratic finite element functions over $\mathcal{T}_h$ i.e.
\begin{align*}
     \b{V^h}&=\{ \b{v} \in [C(\overline \Omega)]^2: \b{v}|_T \in [P_2(T)]^2~~ \forall~ T \in \mathcal{T}_h, \b{v} = \b{0} ~\text{on}~ \Gamma_D \}.
\end{align*}

For concreteness, we state the following discrete trace inequality and inverse inequality which will be used in the subsequent analysis \cite{BScott:2008:FEM}.
\begin{lemma}\label{lem:Discrete}
Let $\b{v}\in [H^1(T)]^2$. Then, 
\begin{align*}
\|\b{v}\|^2_{\b{L^2}(e)} \lesssim h_T^{-1}\|\b{v}\|^2_{\b{L^2}(T)} + h_T \lvert\b{v}\rvert^2_{\b{H^1}(T)},
\end{align*}
where $T \in \mathcal{T}_h$ and $e$ is an edge of $T$.
\end{lemma}
\begin{lemma} 
Let $T~\in~\cT_h$ and $e$ be an edge of $T$. For $\b{v} \in \b{V^h}$, the following estimates hold
\begin{align*}
\begin{aligned}
\|\b{v}\|_{\b{L^2}(e)} &\lesssim h_e^{-\frac{1}{2}}\|\b{v}\|_{\b{L^2}(T)}, \\
\lvert \b{v} \rvert_{\b{H^1}(T)} &\lesssim h_T^{-1} \|\b{v}\|_{\b{L^2}(T)}.
\end{aligned}
\end{align*}

\end{lemma}

\section{Discrete Problem} \label{sec:Discrete}
In this section, we define the discrete formulation of the continuous problem $\eqref{1.2}$. { Further, we  construct an auxiliary discrete space $\b{Q^h}$, where the discrete counterpart of the contact force density is defined which will play a crucial role in forthcoming a posteriori error analysis.}
\par
Let $\{\psi_z\b{e_i}, ~z \in \mathcal{V}_h^o \cup \mathcal{M}_h^o,~i=1,2\}$ represents the canonical nodal Lagrange basis for the space $\b{V^h}$, i.e., for $z \in \mathcal{V}_h^o \cup \mathcal{M}_h^o$
\begin{align*}
    \begin{split}
    \begin{aligned}
    \psi_z(p) = \begin{cases} &1 ~~~\text{if}~ z~=~p \\
    &0~~~~\text{if}~z~\neq~p \end{cases}
    \end{aligned} ~~\forall~ p \in \mathcal{V}_h^o \cup \mathcal{M}_h^o.
    \end{split}
\end{align*}
Note that, for any $\b{v^h}=(v^h_1,v^{h}_2) \in \b{V^h}$ we have the following representation
\begin{align} \label{3.1}
\b{v^h}= \sum_{p \in \mathcal{V}_h^o \cup \mathcal{M}_h^o} \sum_{i=1}^{2}  v^{h}_i(p) \psi_p \b{e_i}.
\end{align}
We define the two discrete subspace $\b{W_1}$ and $\b{W_2}$ of $\b{V^h}$ as

\begin{align*}
    \b{W_1} &= \text{Span} \{\psi_z\b{e_i}~:~z \in (\mathcal{V}_h^o \cup \mathcal{M}_h^o) \setminus (\mathcal{V}_h^C \cup \mathcal{M}_h^C),~i=1,2\},\\
   \b{W_2}&=\text{Span} \{\psi_z\b{e_i}~:~z \in \mathcal{V}_h^C \cup \mathcal{M}_h^C,~i=1,2\}.
\end{align*}
Then, clearly $\b{V^h}= \b{W_1}\bigoplus \b{W_2}$. 
{ It can be observed that the subspace $\b{W_2}$ of $\b{V^h}$ is orthogonal to $\b{W_1}$ with respect to inner product:
\begin{align*}
  \langle {\b{v^h}, \b{w^h} \rangle_{\b{V^h}} := \sum_{T\in \mathcal{T}_h} \frac{|T|}{3}\bigg(\sum_{z\in \mathcal{V}_T} \b{v^h}(z)\b{w^h}(z) + \sum_{z\in \mathcal{M}_T} \b{v^h}(z)\b{w^h}(z)}\bigg)  
\end{align*}
where $\mathcal{V}_T$ and $\mathcal{M}_T$ refers to vertices and midpoints of the element $T$, respectively.} 
\par
\noindent
Further, we introduce the discrete set  $\b{\cK^h}$ of admissible displacements by 
\begin{align*}
    \b{\cK^h} =\{ \b{v^h}=(v^h_1,~v^h_2) \in \b{V^h} ~~\text{s.t}~~ v^h_1(z)\leq 0~~\forall ~z~ \in ~\mathcal{V}_h^C \cup \mathcal{M}_h^C\}.
\end{align*}
The quadratic finite element approximation of \eqref{1.2} is to find $\b{u^h} \in \b{\cK^h}$ such that 
\begin{align}\label{2.1}
    a(\b{u^h}, \b{v^h - u^h}) \geq L(\b{v^h - u^h}) \quad\forall~\b{v^h} \in \b{\cK^h}.
\end{align}
It can be observed that the non-empty, closed and convex set $\b{\cK^h} \not\subset \b{\cK}$ in general. For all $ z \in (\mathcal{V}_h^o \cup \mathcal{M}_h^o) \setminus (\mathcal{V}_h^C \cup \mathcal{M}_h^C)$, observe that $\b{v^h} = \b{u^h} \pm \psi_z\b{e_i} \in \b{\cK^h}$  since $$v^h_1(p) = u^h_1(p)
\leq 0~ \forall~p~\in \mathcal{V}_h^C \cup \mathcal{M}_h^C.$$
Therefore, we find
\begin{align}\label{2.2}
  a(\b{u^h}, {\psi_z\b{e_i}})~ = ~L({\psi_z\b{e_i}}) ~\forall~ z \in (\mathcal{V}_h^o \cup \mathcal{M}_h^o) \setminus (\mathcal{V}_h^C \cup \mathcal{M}_h^C). 
\end{align}
Henceforth,
\begin{align}\label{2.3}
  a(\b{u^h}, \b{v^h})~ = ~L({\b{v^h}}) \quad~\forall~ \b{v^h}~\in \b{W_1}.  
\end{align}
Further, for $z \in ~(\mathcal{V}_h^C \cup \mathcal{M}_h^C)$, we observe that $\b{v^h} = \b{u^h} - \psi_z\b{e_1} \in \b{\cK^h}$ as
$$v^h_1(p) = \begin{cases} u^h_1(p)~~~~~~~ &p\neq z \\
u^h_1(p)-1~~~~~ &p=z \end{cases}
~~~~\leq 0 \quad\forall~p~\in \mathcal{V}_h^C \cup \mathcal{M}_h^C. $$
On the similar lines, one can verify $\b{v^h} = \b{u^h} \pm \psi_z\b{e_2} \in \b{\cK^h}~ \forall~z\in(\mathcal{V}_h^C \cup \mathcal{M}_h^C)$. Thus,
\begin{align} \label{2.4}
    \begin{split}
    \begin{aligned}
     a(\b{u^h}, {\psi_z\b{e_1}})~ &\leq ~L({\psi_z\b{e_1}})\\
     a(\b{u^h}, {\psi_z\b{e_2}})~ &= ~L({\psi_z\b{e_2}})
    \end{aligned} ~~~~~ \forall~z~\in~(\mathcal{V}_h^C \cup \mathcal{M}_h^C).
    \end{split}
    \end{align}
\\
We now proceed to introduce a discrete space where we can define discrete counterpart of the contact force density $\b{\lambda}$. The
construction of the discrete space requires the introduction of some more notations related to the contact zone. Let $\mathcal{T}_h^C$ denotes the mesh formed by the edges of $\mathcal{T}_h$ on $\Gamma_C$ which is characterized by the subdivision of ($\b{x}_i^c)_{0\leq i\leq n}$ where $(\b{x}_i^c)_{0\leq i\leq n} \in \mathcal{V}_h^C$. Let $t_i = [\b{x}_i^c,~\b{x}_{i+1}^c]_{0 \leq i \leq n-1}$ denotes the element on $\Gamma_C$ with the midpoint $m_i^c$. Hence, we can write each element $t_i$ as union of two sub interval $q_i^1 \cup q_i^2$ where $q_i^1~=~[\b{x}_i^c,~\b{m}_{i}^c]$ and $q_i^2~=~[\b{m}_i^c,~\b{x}_{i+1}^c]$.  Thus, we can rewrite 
\begin{align*}
    \Gamma_C =\underset{0\leq i \leq n-1}{\bigcup}  q_i^1 \cup q_i^2.
\end{align*}
Now, with the following notations we define the discrete space $\b{Q^h}$ as
\begin{align} \label{qh}
\b{Q^h} =\{ \b{v^h}~\in~[C(\overline \Gamma_C)]^2: \b{v^h}|_{q_i^j} \in [P_1(q_i^j)]^2,~1\leq i \leq n-1,~j~=1,2\}.
\end{align}
We observe that the  dimension of the space $\b{Q^h}$ is $2\lvert {\cV}_h^C \cup \mathcal{M}_h^C \rvert$. Let $\{\phi_z\b{e_i}~: z\in\mathcal{V}_h^C \cup \mathcal{M}_h^C\}$ be the canonical nodal Lagrange basis for $\b{Q^h}$, i.e., for $z \in \mathcal{V}_h^C \cup \mathcal{M}_h^C$
\begin{align*}
    \begin{split}
    \begin{aligned}
    \phi_z(p) = \begin{cases} &1 ~~~\text{if}~ z~=~p \\
    &0~~~~\text{if}~z~\neq~p \end{cases}
    \end{aligned} \quad\forall~p \in \mathcal{V}_h^C \cup \mathcal{M}_h^C.
    \end{split}
\end{align*}
Define a linear map $\pi_h~:~\b{Q^h} \longrightarrow \b{W_2}$ by
\begin{align} \label{phih}
    \pi_h\b{v^h}~:=\sum_{z \in \mathcal{V}_h^C \cup \mathcal{M}_h^C}\sum_{i=1}^{2}  v^h_i(z)\psi_z\b{e_i}~~\quad\forall~\b{v^h}=(v^h_1,v^h_2) \in \b{Q^h}.
    \end{align}
Clearly, the map $\pi_h$ is well defined and one-one. Since dimension of space $\b{Q^h}$ and $\b{W_2}$ are equal, therefore the map $\pi_h$ is bijective and hence $\pi_h^{-1}~:~\b{W_2} \longrightarrow \b{Q^h}$ exists and is given by
\begin{align*}
    \pi_h^{-1}\b{v^h}~=\sum_{z \in \mathcal{V}_h^C \cup \mathcal{M}_h^C}\sum_{i=1}^{2}  v^h_i(z)\phi_z\b{e_i}~~\quad\forall~\b{v^h}=(v^h_1,v^h_2) \in \b{W_2}.
    \end{align*}
It can be observed that 
 \begin{align}\label{match}
    \pi_h^{-1}\b{v^h}(z)~= \b{v^h}(z)~\forall~{z \in \mathcal{V}_h^C \cup \mathcal{M}_h^C}~ \quad\forall~\b{v^h} \in \b{W_2}.
    \end{align}   

\par
Now, we turn our attention to introduce the discrete contact force density $\b{\lambda^h} \in \b{Q^h}$ which is defined as 
\begin{align}\label{2.8}
\langle \b{\lambda^h}, \b{v^h} \rangle_{\b{h}} = L(\pi_h\b{v^h}) - a(\b{u^h}, \pi_h\b{v^h})~\quad\forall~ \b{v^h}\in\b{Q^h},
\end{align}
where, the inner product $\langle \cdot, \cdot \rangle_{\b{h}}$ on the space $\b{Q^h}$ is defined as
\begin{align*}
  \langle \b{w^h}, \b{v^h} \rangle_{\b{h}}~:= \sum_{z \in \mathcal{V}_h^C \cup \mathcal{M}_h^C} \b{w^h}(z) \cdot \b{v^h}(z) \int_{\gamma_{z,C}} \phi_z~ds .
\end{align*}
Note that, $\b{\lambda^h}$ is well-defined since $\langle \cdot, \cdot \rangle_{\b{h}} $ defines an inner product on $\b{Q^h}$.
In the following lemma, we will establish the properties of discrete contact force density $\b{\lambda^h} \in \b{Q^h}$.
\begin{lemma}\label{lem:disL} The discrete contact force density $\b{\lambda^h}=(\lambda^h_1,~\lambda^h_2) \in \b{Q^h}$ satisfies the following sign properties.
\begin{align*}
    \begin{split}
    \begin{aligned}
\lambda^h_1(p) ~&~\geq 0 \quad \forall p \in \mathcal{V}_h^C \cup \mathcal{M}_h^C,\\
\lambda^h_2(p)~ &=~ 0 \quad \forall p \in \mathcal{V}_h^C \cup \mathcal{M}_h^C.
    \end{aligned}
    \end{split}
\end{align*}

\end{lemma}
\begin{proof}
The proof of this lemma follows by suitable construction of a test function $\b{v^h} \in \b{Q^h}$. Let $p \in \mathcal{V}_h^C \cup \mathcal{M}_h^C$ be an arbitrary node. We choose a test function $\b{v^h}$ as follows
\begin{align*}
    \begin{split}
    \begin{aligned}
    \b{v^h}(z) = \begin{cases} &(1,0) ~~~\text{if}~ z~=~p, \\
    &(0,~0)~~\text{if}~z ~\neq~ p, \end{cases}
    \end{aligned} \quad \forall z \in \mathcal{V}_h^C \cup \mathcal{M}_h^C.
    \end{split}
\end{align*}
Further, using the definition of $\pi_h$, we have
\begin{align*}
    \pi_h\b{v^h} =\sum_{z \in \mathcal{V}_h^C \cup \mathcal{M}_h^C}(v^h_1(z)\psi_z,~v^h_2(z)\psi_z)=(\psi_p,~0)
    =\psi_p\b{e_1}.
    \end{align*}
Thus, the use of \eqref{2.4} yields
\begin{align}\label{2.9}
    \langle \b{\lambda^h}, \b{v^h} \rangle_{\b{h}}  &= L(\pi_h\b{v^h}) - a(\b{u^h}, \pi_h\b{v^h}) \nonumber \\
    &= L(\psi_p\b{e_1}) - a(\b{u^h}, \psi_p\b{e_1}) \nonumber\\
    &\geq 0,
\end{align}
whereas, using the definition of $\langle \cdot, \cdot \rangle_{\b{h}}$, we find
\begin{align}\label{2.10}
 \langle \b{\lambda^h}, \b{v^h} \rangle_{\b{h}} &= \sum_{z \in \mathcal{V}_h^C \cup \mathcal{M}_h^C} \b{\lambda^h}(z) \cdot \b{v^h}(z) \int_{\gamma_{z,C}} \phi_z~ds\nonumber\\
 &= \b{\lambda^h}(p) \cdot \b{v^h}(p) \int_{\gamma_{p,C}} \phi_p~ds\nonumber\\
 &= {\lambda^h_1}(p)\int_{\gamma_{p,C}} \phi_p~ds .
\end{align}
Combining $\eqref{2.9}$, $\eqref{2.10}$ and taking into account $\int_{\gamma_{p,C}} \phi_p~ds > 0$, we find $\lambda^h_1(p)\geq 0$. Since $p \in \mathcal{V}_h^C \cup \mathcal{M}_h^C$ is arbitrary, it follows that $\lambda^h_1(p)\geq 0~\forall ~p \in \mathcal{V}_h^C \cup \mathcal{M}_h^C$. Analogously for any $p \in \mathcal{V}_h^C \cup \mathcal{M}_h^C$, we define $\b{v^h} \in \b{Q^h}$ such that 
\begin{align*}
\begin{split}
\begin{aligned}
\b{v^h}(z) = \begin{cases} &(0,1) ~~~\text{if}~ z=p \\
&(0,~0)~~\text{if}~z\neq p \end{cases}
\end{aligned} \quad \forall z \in \mathcal{V}_h^C \cup \mathcal{M}_h^C.
\end{split}
\end{align*}
In this case, we have $\pi_h\b{v^h} = \psi_p\b{e_2}$. Therefore, using \eqref{2.4}  we have
\begin{align}\label{3.111}
   \langle \b{\lambda^h}, \b{v^h} \rangle_{\b{h}}  &= L(\pi_h\b{v^h}) - a(u^h, \pi_h\b{v^h}) \nonumber \\
    &= L(\psi_p\b{e_2}) - a(u^h, \psi_p\b{e_2}) \nonumber\\
    &= 0,
\end{align}
and
\begin{align}\label{2.12}
\langle \b{\lambda^h}, \b{v^h} \rangle_{\b{h}} &= \sum_{z \in \mathcal{V}_h^C \cup \mathcal{M}_h^C} \b{\lambda^h}(z) \cdot \b{v^h}(z) \int_{\gamma_{z,C}} \phi_z~ds \nonumber\\
&= \b{\lambda^h}(p) \cdot \b{v^h}(p)\int_{\gamma_{p,C}} \phi_p~ds \nonumber\\
&= {\lambda^h_2}(p)\int_{\gamma_{p,C}} \phi_p~ds.
\end{align} 
Using equation \eqref{3.111} and $ \int_{\gamma_{p,C}} \phi_p> 0$, it follows  ${\lambda^h_2}(p)=0.$ Consequently, it follows   ${\lambda^h_2}(p)=0~\forall~p \in \mathcal{V}_h^C \cup \mathcal{M}_h^C$.
\end{proof}

In order to carry out further analysis, we define the linear residual $\b{R^{lin}} \in \b{V^*}$ as
\begin{align} \label{LRES}
\b{\langle \b{R^{lin}}, \b{\phi} \rangle_{-1,1}}:= L(\b{\phi}) -a (\b{u^h}, \b{\phi}) \quad \forall \b{\phi} \in \b{V}.
\end{align}
\par
\noindent
For any $\b{\phi}=(\phi_1,~\phi_2)\in \b{V}$, the linear residual can be represented as 
\begin{align*}
\b{\langle \b{R^{lin}}, \b{\phi} \rangle_{-1,1}} = \sum_{i=1}^2 \langle {R^{lin}_i}, {\phi_i} \rangle_{-1,1},
\end{align*}
where,
\begin{align*}
  \langle {R^{lin}_1}, {\phi_1} \rangle_{-1,1} &:= L((\phi_1,0)) -a (\b{u^h}, (\phi_1,0)),\\
  \langle R^{lin}_2, {\phi_2} \rangle_{-1,1} &:= L((0,\phi_2)) -a (\b{u^h},(0,\phi_2)).
\end{align*}
Further, for any $\b{\phi^h} \in \b{V^h}$, we have
\begin{align}
\b{\langle \b{R^{lin}}, \b{\phi^h} \rangle_{-1,1}} = L(\b{\phi^h}) -a (\b{u^h}, \b{\phi^h}) \quad \forall~ \b{\phi^h} \in \b{V^h}.
 \end{align}

In particular, we assume $\b{\phi^h}= \psi_z \b{e_i}$ for $z \in \mathcal{V}_h^C \cup \mathcal{M}_h^C$ to derive
\begin{align} \label{3.16}
\langle \b{R^{lin}},  \psi_z \b{e_i} \rangle_{\b{-1,1}} = L( \psi_z \b{e_i}) -a (\b{u^h},  \psi_z \b{e_i}).
\end{align} 
Using equation \eqref{phih}, we have $ \pi_h(\phi_z \b{\b{e_i}})= \psi_z \b{e_i}$ for $i=1,2.$ Finally, using the equations (\ref{2.8}) and (\ref{3.16}), we have the following relation for any $z \in \mathcal{V}_h^C \cup \mathcal{M}_h^C$
\begin{align} \label{3.17}
\langle \b{R^{lin}},  \psi_z \b{e_i} \rangle_{\b{-1,1}} &= L(  \pi_h(\phi_z \b{e_i})) -a (\b{u^h}, \pi_h(\phi_z \b{e_i})), \nonumber \\ & = \langle \b{\lambda^h}, \phi_z\b{e_i} \rangle_{\b{h}}.
\end{align} 
The above relation between $\b{\lambda^h}$ and $\b{R^{lin}}$ plays a key role in later analysis. Let $\b{v^h}=(v^h_1,v^h_2) \in \b{V^h}$, using integration by parts and equation (\ref{3.1}), we find
\begin{align} \label{residual}
\b{\langle \b{R^{lin}},\b{v^h} \rangle_{-1,1}} &=\sum_{i=1}^{2}  \sum_{p \in \mathcal{V}_h^o \cup \mathcal{M}_h^o} \big[L(v^{h}_i(p) \psi_p \b{e_i}) - a(\b{u^h},v^{h}_i(p) \psi_p \b{e_i} ) \big] \nonumber \\ & = \sum_{i=1}^{2}  \sum_{p \in \mathcal{V}_h^o \cup \mathcal{M}_h^o} \int_{\omega_p} (\b{f} + \b{div\sigma (u^h)}) \cdot v^{h}_i(p) \psi_p \b{e_i} ~dx\nonumber \\ & \hspace{0.3cm}- \sum_{i=1}^{2}  \sum_{p \in \mathcal{V}_h^o \cup \mathcal{M}_h^o} \int_{\gamma_{p,I}} \sjump{\b{\sigma(u^h)}} \cdot v^{h}_i(p) \psi_p \b{e_i} ~ds\nonumber \\ & \hspace{0.3cm}+ \sum_{i=1}^{2}  \sum_{p \in \mathcal{V}_h^N \cup \mathcal{M}_h^N} \int_{\gamma_{p,N}} (\b{g}-\b{\sigma(u^h)} \b{n}) \cdot v^{h}_i(p) \psi_p \b{e_i} ~ds\nonumber\\ & \hspace{0.3cm}- \sum_{i=1}^{2}  \sum_{p \in \mathcal{V}_h^C \cup \mathcal{M}_h^C} \int_{\gamma_{p,C}} \b{\sigma(u^h)} \b{n} \cdot v^{h}_i(p) \psi_p \b{e_i}~ds.
\end{align}
Using the equations (\ref{2.3}), (\ref{2.4}), \eqref{2.8} and (\ref{3.17}), we derive important characterizations for $\b{R^{lin}}$
\begin{align}
\langle \b{R^{lin}},  \psi_z \b{e_i} \rangle_{\b{-1,1}} &= 0 \quad \forall z \in  (\mathcal{V}_h^o  \cup \mathcal{M}_h^o) \setminus (\mathcal{V}_h^C \cup \mathcal{M}_h^C),~i=1,2, \label{prop1}\\ \langle \b{R^{lin}},  \psi_z \b{e_2} \rangle_{\b{-1,1}} &=0 \quad \forall z \in  \mathcal{V}_h^C \cup \mathcal{M}_h^C.\label{prop2}
\end{align}
In the subsequent analysis, for the ease of the presentation we abbreviate the interior residual as $\b{r(u^h)} = \b{f} + \b{div} \b{\sigma(u^h)}$. Further, the jump terms which are either the difference between the contact stresses of two neighboring elements or the difference between Neumann data and boundary stress at Neumann boundary or the boundary stresses at contact boundary are abbreviated as 
\begin{itemize}
\item For $e\in \mathcal{E}_h^i$
\begin{align*}
    \b{J^I}(\b{u^h}) := \sjump{\b{\sigma(\b{u^h)}}},
\end{align*}
\item For $e\in \mathcal{E}_h^N$
\begin{align*}
    \b{J^N}(\b{u^h}) := \b{g}-\b{\sigma(u^h)} \b{n},
\end{align*} 
\item For $e\in \mathcal{E}_h^C$
\begin{align*}
\b{J^C_{tan}}(\b{u^h}) := \hat{\b{\sigma_2}}(\b{u^h}).~~~
\end{align*} 
\end{itemize}
\section{Quasi Discrete Contact Force Density} \label{sec4}
\par
In this section, we introduce the quasi discrete contact force density which imitates the properties of continuous contact force density $\b{\lambda}$ but computed using the discrete solution and discrete contact force density. For any $p \in \mathcal{V}_h^C \cup \mathcal{M}_h^C$, we take
the node values of a discrete contact force density obtained by lumping the boundary mass matrix and define
\begin{align}\label{def:sp}
\b{s_p}= (s^1_p,s^2_p),
\end{align}
where $s^1_p:= \frac{\langle \b{\lambda^{h}}, \phi_p \b{e_1} \rangle_{\b{h}}}{ \int_{\gamma_{p,C}} \phi_p~ds}= \lambda^{h}_1(p)$ and $s^2_p:=\lambda^{h}_2(p)=0.$ Next, with the help of \eqref{def:sp}, we introduce the quasi discrete contact force density $\b{\tilde{\lambda}^{h}} \in \b{V^*}$ in the following way
\begin{align} \label{3.30}
\b{\langle \b{\tilde{\lambda}^{h}} , \b{v} \rangle_{-1,1}} := \sum_{i=1}^{2} \langle \tilde{\lambda}^{h}_i , v_i \rangle_{-1,1} \quad \forall \b{v} \in \b{V},
\end{align}
where, for $i=1,2$
\begin{align}\label{aaaa}
\langle \tilde{\lambda}^{h}_i , v_i \rangle_{-1,1} := \int_{\Gamma_C} \bigg(\sum_{z \in \mathcal{V}_h^C \cup \mathcal{M}_h^C} s^i_z \phi_z \bigg) v_i~ds.
\end{align}
We derive the sign property for $\b{\tilde{\lambda}^{h}}$ and a useful preliminary result in the next two lemmas, respectively.
\begin{lemma} \label{sign1}
It holds that
\begin{align}
\b{\langle \b{\tilde{\lambda}^{h}} , \b{v} \rangle_{-1,1}} = \langle \tilde{\lambda}^{h}_1, v_1 \rangle_{-1,1} \geq 0 \quad \text{whenever}~~ v_1 \geq 0,
\end{align}
where $\b{v}=(v_1,v_2) \in \b{V}.$
\begin{proof}
Using the definition (\ref{3.30}) and the fact that $s^2_p=0~ \forall~p \in \mathcal{V}_h^C \cup \mathcal{M}_h^C  $, we obtain
\begin{align*}
\langle \b{\tilde{\lambda}^{h}} , \b{v} \rangle_{\b{-1,1}} = \langle \tilde{\lambda}^{h}_1, v_1 \rangle_{-1,1}.\end{align*}
 Next, we will prove for any $v_1 \geq 0$, we have $\langle \tilde{\lambda}^{h}_1, v_1 \rangle_{-1,1} \geq 0.$ 
 Using Lemma \ref{lem:disL}, we deduce that $s^1_p = \lambda^h_1(p) \geq 0$ and the equation (\ref{3.17}) yields
 \begin{align}\label{i0i}
0 \leq  s^1_p &= \frac{\langle \b{\lambda^{h}}, \phi_p \b{e_1} \rangle_{\b{h}}}{ \int_{\gamma_{p,C}} \phi_p~ds} = \frac{\langle \b{R^{lin}},  \psi_p \b{e_1} \rangle_{\b{-1,1}} }{ \int_{\gamma_{p,C}} \phi_p~ds}.
 \end{align}
 Moreover, from \eqref{aaaa} we deduce that
 \begin{align} \label{3.23}
 \langle \tilde{\lambda}^{h}_1, v_1 \rangle_{-1,1} = \sum_{p \in \mathcal{V}_h^C \cup \mathcal{M}_h^C}  \langle \b{R^{lin}},  \psi_p \b{e_1} \rangle_{\b{-1,1}}  c_p(v_1),
  \end{align}where $c_p(v_1)= \frac{\int_{\gamma_{p,C}} v_1 \phi_p~ds}{\int_{\gamma_{p,C}} \phi_p~ds }$. 
  Since $c_p(v_1) \geq 0 ~\forall~ v_1 \geq 0$, thus combining \eqref{i0i} and \eqref{3.23} yields the desired result.
\end{proof}
\end{lemma}
\begin{lemma} \label{sign2}
The following holds
\begin{align}
\langle \tilde{\lambda}^{h}_1 , v_1 \rangle_{-1,1} = \sum_{p \in \mathcal{V}_h^C \cup \mathcal{M}_h^C} \langle \tilde{\lambda}^{h}_1, v_1 \psi_p \rangle_{-1,1},
\end{align}
where $\b{v}=(v_1,v_2) \in \b{V}.$
\end{lemma}
\begin{proof}
Using the partition of unity \cite{BScott:2008:FEM}, we have 
\begin{align}\label{131}
    \langle \tilde{\lambda}^{h}_1, v_1 \rangle_{-1,1} &= \sum_{p \in \mathcal{V}_h^o \cup \mathcal{M}_h^o}\langle \tilde{\lambda}^{h}_1, v_1\psi_p \rangle_{-1,1} \nonumber\\
    &=\sum_{p \in \mathcal{V}_h^o \cup \mathcal{M}_h^o}\int_{\Gamma_C} \bigg(\sum_{z \in \mathcal{V}_h^C \cup \mathcal{M}_h^C} s^1_z \phi_z \bigg) v_1\psi_p~ds.
\end{align}
Since $\psi_p$ on $\Gamma_C$ will be non-zero only for the nodes on $\overline{\Gamma_C}$.
Thus, \eqref{131} reduces to 
\begin{align*}
    \langle \tilde{\lambda}^{h}_1, v_1 \rangle_{-1,1} &= \sum_{p \in \mathcal{V}_h^C \cup \mathcal{M}_h^C}\int_{\Gamma_C} \bigg(\sum_{z \in \mathcal{V}_h^C \cup \mathcal{M}_h^C} s^1_z \phi_z \bigg) v_1\psi_p~ds, \\
    &= \sum_{p \in \mathcal{V}_h^C \cup \mathcal{M}_h^C} \langle \tilde{\lambda}^{h}_1, v_1 \psi_p \rangle_{-1,1}.
\end{align*}
\end{proof}
Next, we categorize  actual contact nodes $p \in \mathcal{V}_h^C \cup \mathcal{M}_h^C$ \big($u^{h}_1(p)=0$\big) in two different categories.
\begin{enumerate}
\item Full contact nodes $\mathcal{N}^{FC}_h:=\{p \in \mathcal{V}_h^C \cup \mathcal{M}_h^C~|~u^{h}_1= 0~\text{on}~\gamma_{p,C}\}$.
\item The remaining actual contact nodes are called semi contact nodes and denoted by $\mathcal{N}^{SC}_h$.
\end{enumerate}
\par
\noindent
Denote $\mathcal{N}_h^{NC}$ as the set of no actual contact nodes, i.e., for $p \in \mathcal{N}_h^{NC}$, $u^{h}_1(p) \neq 0$.
\par
Next, we derive an important property of $\b{\tilde{\lambda}^{h}}$ with the help of upcoming lemma.
\begin{lemma} \label{sign}
It holds that
\begin{align} \label{signn}
\langle \b{\lambda^{h}}, \phi_p \b{e_1} \rangle_{\b{h}} =0 \quad \forall p \in \mathcal{N}_h^{NC}.
\end{align}
\end{lemma}
\begin{proof}
Let $p$ be any non actual contact node, i.e., we have $u^{h}_1(p) < 0$. We note for sufficiently small $\kappa>0$ such that $ 0< \kappa < -{u^h_1(p)}$, we have $\b{v^h} = \b{u^h}+ \kappa \psi_p\b{e_1} \in \b{\cK^h}$. Using \eqref{2.1}, we conclude
\begin{align}
a(\b{u^h}, {\psi_p\b{e_1}})~ \geq ~L({\psi_p\b{e_1}}).
\end{align}
Finally, we have $\langle \b{\lambda^h}, \phi_p\b{e_1} \rangle_{\b{h}}=0$ in view of \eqref{3.17} and \eqref{2.4}.
\end{proof}
{
\begin{remark}\label{rem}
From the last lemma, we deduce that $s^1_p := \frac{\langle \b{\lambda^{h}}, \phi_p\b{e_1}\rangle_{\b{h}}}{\int_{\gamma_{p,C}}\phi_p~ds} = 0~~\forall p \in \mathcal{N}_h^{NC}.$ Thus, using \eqref{aaaa} we obtain,
\begin{equation*}
\langle \tilde{\lambda}^{h}_1 , v_1 \rangle_{-1,1} = \int_{\Gamma_C} \bigg( \underset{z \in \mathcal{V}_h^C \cup \mathcal{M}_h^C}{\sum} s^1_z \phi_z \bigg) v_1~ds \\ = \int_{\Gamma_C} \bigg(\underset{z \in \mathcal{N}^{FC}_h \cup \mathcal{N}^{SC}_h}{\sum}  s^1_z \phi_z\bigg) v_1 ~ds,   
\end{equation*}
for any $\b{v}=(v_1,v_2) \in \b{V}.$
\end{remark}}
Further, for $\b{\upsilon}=(\upsilon_1, \upsilon_2)\in \b{V}$, we will introduce the constants $c_p(\upsilon_i)$ for any $p \in \mathcal{V}_h \cup \mathcal{M}_h$ defined such that they fulfills $L^2$ approximation properties. For all the non-contact nodes and all contact nodes with $i\neq 1$, we define the constants 
\begin{align}\label{1111}
    c_p(\upsilon_i)= \dfrac{\int_{\omega_p}{\upsilon_i\psi_p~dx}}{\int_{\omega_p}{\psi_p~dx}}
\end{align}
and for semi contact and full contact nodes, the constants $c_p(\upsilon_1)$ are chosen such that
\begin{align}\label{2222}
    c_p(\upsilon_1)= \dfrac{\int_{\tilde{\gamma}_{p,C}}{\upsilon_1\phi_p~ds}}{\int_{\tilde{\gamma}_{p,C}}{\phi_p~ds}}
\end{align}
where $\tilde{\gamma}_{p,C}$ is a proper subset of ${\gamma}_{p,C}$ such that it contains $p$ and for any two different nodes $p_1$ and $p_2$ in $\gamma_{p,C}$,  $\tilde{\gamma}_{p_1,C} \cap \tilde{\gamma}_{p_2,C} = \phi.$ These constants are helpful in deriving the lower bound of the error estimator. Also, we have the following approximation properties \cite{Krause:2015:apost_Sig}.
\begin{lemma}\label{sign3}
Let $\upphi$ be an arbitrary function in $H^1(\Omega)$ and $c_p(\upphi)$ be one of the mean values defined in \eqref{1111} and \eqref{2222}. Then, we have the following $L^2$-approximation properties.
\begin{align*}
\|\upphi - c_p(\upphi)\|_{L^2(\omega_p)} &\lesssim h_p\|\nabla \upphi\|_{L^2(\omega_p)}, \\
\|\upphi - c_p(\upphi)\|_{L^2(\gamma_p)} &\lesssim h_p^{1/2}\|\nabla \upphi\|_{L^2(\omega_p)},
\end{align*}
where $\gamma_p$ is either $\gamma_{p,I}, \gamma_{p,N} \text{~or~} \gamma_{p,C}.$
\end{lemma}
\section{A posteriori Error Analysis}\label{sec:Apoestriori}
In this section, we turn our attention to analyze reliability and efficiency of a posteriori error estimator. We begin by introducing the following contributions of the error estimator
\begin{align*}
    \eta_1 :=&\bigg(\sum_{p \in \mathcal{V}_h^o \cup \mathcal{M}_h^o} \eta_{1,p}^2\bigg)^{1/2},~~~~~~~\eta_{1,p}:=h_p\|\b{f} + \b{div} \b{\sigma}(\b{u^h})\|_{\b{L^2}(\omega_p)},\\
    \eta_2 :=&\bigg(\sum_{p \in \mathcal{V}_h^o \cup \mathcal{M}_h^o} \eta_{2,p}^2\bigg)^{1/2},~~~~~~\eta_{2,p}:=h_p^{\frac{1}{2}}\|~\sjump{\b{\sigma}(\b{u^h})}~ \|_{\b{L^2}(\gamma_{p,I})},\\
    \eta_3 :=&\bigg(\sum_{p \in \mathcal{V}_h^N \cup \mathcal{M}_h^N} \eta_{3,p}^2\bigg)^{1/2},~~~~\eta_{3,p}:= h_p^{\frac{1}{2}} \|\b{\sigma}(\b{u^h})\b{n}- \b{g}\|_{\b{L^2}(\gamma_{p,N})},\\
    \eta_4 :=&\bigg(\sum_{p \in \mathcal{V}_h^C \cup \mathcal{M}_h^C} \eta_{4,p}^2\bigg)^{1/2},~~~~\eta_{4,p}:=
     h_p^{\frac{1}{2}}\|\hat{\sigma}_2(\b{u^h})\|_{\b{L^2}(\gamma_{p,C})},\\
    \eta_5 :=&\bigg(\sum_{p \in \mathcal{V}_h^C \cup \mathcal{M}_h^C} \eta_{5,p}^2\bigg)^{1/2},~~~~\eta_{5,p}:=
     h_p^{\frac{1}{2}}\|\hat{\sigma}_1(\b{u^h})\|_{\b{L^2}(\gamma_{p,C})},\\
    \eta_6 :=&\bigg(\sum_{p \in \mathcal{N}^{SC}_h} \eta_{6,p}^2\bigg)^{1/2},~~~~~~~~~~\eta_{6,p}:=(s^1_pd_p)^{\frac{1}{2}}, \\
    \eta_7 := &~\|(u^h_1)^{+}\|_{H^{\frac{1}{2}}(\Gamma_C)}, \\
    \end{align*}
    where $d_p:= \int_{\tilde{\gamma}_{p,C}} (- u^h_1)^+\phi_p~ds$.
Let $\eta_h$ denotes the total residual estimator and is defined by 
   \begin{align} \label{esti}
    \eta_h^2=\eta_1^2+\eta_2^2+\eta_3^2+\eta_4^2+\eta_5^2+\eta_6^2+\eta_7^2.
   \end{align}

The following subsection guarantees the reliability of the error estimator $\eta_h$.
\subsection{Reliability of the error estimator}
\par
\noindent
Define the Galerkin functional $\b{G_h} : \b{V} \longrightarrow \mathbb{R}$ by 
\begin{align}\label{4.1}
    \b{G_h}(\b{v}) := a(\b{u}-\b{u^h},~\b{v}) + \b{\langle \b{\lambda}- \b{\tilde{\lambda}^{h}},\b{v} \rangle_{-1,1}} ~\forall~ \b{v}~\in~{\b{V}}.
\end{align}
In the following lemma, we will observe the relation {between the true error and the Galerkin functional}.
\begin{lemma} \label{lem:Rel}
It holds that
\begin{align*}
 \|\b{u-u^h}\|^2_{\b{H^1(\Omega)}} + \|\b{\lambda}- \b{\tilde{\lambda}^{h}}\|^2_{\b{H^{-1}(\Omega)}} \leq C_1 \|\b{G_h}\|^2_{\b{V^*}} + C_2 \b{\langle \b{\tilde{\lambda}^{h}}- \b{\lambda}, \b{u-u^h} \rangle_{-1,1}},
\end{align*}
where $C_1$ and $C_2$ are generic constants.
\end{lemma}
\begin{proof}
Using the $\b{V}$- ellipticity of the bilinear form $a( \cdot,\cdot)$ and  equation $\eqref{4.1}$, we find
\begin{align*}
  \alpha \|\b{u-u^h}\|^2_{\b{H^1(\Omega)}} &\leq a(\b{u-u^h},~\b{u-u_h})\\
   &=  \b{G_h}(\b{u-u^h})  + \b{\langle \b{\tilde{\lambda}^{h}}- \b{\lambda}, \b{u-u^h} \rangle_{-1,1}}  \\
   &\lesssim \|\b{G_h}\|_{\b{V^*}}\|\b{u-u^h}\|_{\b{H^1(\Omega)}}+ \b{\langle \b{\tilde{\lambda}^{h}}- \b{\lambda}, \b{u-u^h} \rangle_{-1,1}}~.
\end{align*}
A use of Young's inequality in the last equation yields
\begin{align}\label{eq:Int_uh}
   \alpha\|\b{u-u_h}\|^2_{\b{H^1(\Omega)}} \lesssim  \frac{1}{2\alpha}\|\b{G_h}\|^2_{\b{V^*}} + \frac{\alpha}{2}\|\b{u-u^h}\|^2_{\b{H^1(\Omega)}}+ \b{\langle \b{\tilde{\lambda}^{h}}- \b{\lambda}, \b{u-u^h} \rangle_{-1,1}}~,
\end{align}
for some positive constant $\alpha$. As a result, we obtain a bound on $\|\b{u-u^h}\|^2_{H^1(\Omega)}$.  Further, using
\begin{align*}
\|\b{\lambda}- \b{\tilde{\lambda}^{h}}\|_{\b{V^{*}}} :=  \sup_{\b{\phi} \in \b{V}} \frac{ \b{G_h}(\b{\phi})- a(\b{u}-\b{u^h},~\b{\phi})}{\|\b{\phi}\|_{\b{V}}},
\end{align*} 
together with the bound on $\|\b{u-u^h}\|_{\b{H^1(\Omega)}}$ given in \eqref{eq:Int_uh} and the continuity of the bilinear form, we obtain the bound for  $\|\b{\lambda}- \b{\tilde{\lambda}^{h}}\|_{\b{H^{-1}(\O)}}$.
\end{proof}
In the next lemma, we infer the relation between the functional $\b{G_h}$ and estimator $\eta_h$.
\begin{lemma} \label{lem:Gh}
It holds that
\begin{align*}
    \|\b{G_h}\|_{\b{V^*}}\lesssim \eta_h.
\end{align*}
\end{lemma}
\begin{proof}
For any $\b{v}\in \b{V}$, we have
\begin{align}
   \b{G_h}(\b{v}) &= a(\b{u}-\b{u^h},~\b{v}) + \b{\langle \b{\lambda}- \b{\tilde{\lambda}^{h}},\b{v} \rangle_{-1,1}} \hspace{1cm} (\text{using equation (\ref{4.1})}) \nonumber \\ &= {L}(\b{v}) - a (\b{u^h}, \b{v}) - \b{\langle \b{\tilde{\lambda}^{h}},\b{v} \rangle_{-1,1}} \hspace{1cm} (\text{using equation (\ref{CFD})})  \nonumber  \\ & = \b{\langle \b{R^{lin}}, \b{v} \rangle_{-1,1}} - \b{\langle \b{\tilde{\lambda}^{h}},\b{v} \rangle_{-1,1}}\hspace{1.7cm} (\text{using equation (\ref{LRES})})  \nonumber  \\ & = \sum_{i=1}^{2}  \sum_{p \in \mathcal{V}_h^o \cup \mathcal{M}_h^o}  \langle R^{lin}_i, v_i \psi_p \rangle_{-1,1} - \langle \tilde{\lambda}^{h}_1, v_1 \rangle_{-1,1}  \hspace{1.2cm} (\text{using Lemma (\ref{sign2}) and Lemma (\ref{sign1}))}\nonumber  \\ & = \sum_{i=1}^{2}  \sum_{p \in (\mathcal{V}_h^o \cup \mathcal{M}_h^o) \setminus (\mathcal{V}_h^C \cup  \mathcal{M}_h^C)   }  \langle R^{lin}_i, v_i \psi_p \rangle_{-1,1} + \sum_{p \in \mathcal{V}_h^C \cup \mathcal{M}_h^C}  \langle R^{lin}_2, v_2 \psi_p \rangle_{-1,1}  \nonumber  \\ & \hspace{0.3cm}+ \sum_{p \in \mathcal{V}_h^C \cup \mathcal{M}_h^C}  \langle R^{lin}_1, v_1 \psi_p \rangle_{-1,1} - \langle \tilde{\lambda}^{h}_1, v_1 \rangle_{-1,1}. \label{4.2}
\end{align}
Using the constants $c_p(v_i)$ introduced in the section \ref{sec4} together with equations (\ref{prop1}) and (\ref{prop2}), for $i=1,2$, we derive
\begin{align}
c_p(v_i)\b{\langle R^{lin}},  \psi_p \b{e_i} \rangle_{\b{-1,1}} &= 0 \quad \forall p \in  (\mathcal{V}_h^o  \cup \mathcal{M}_h^o) \setminus (\mathcal{V}_h^C \cup \mathcal{M}_h^C), \label{prop11}
\\
c_p(v_2)\langle \b{R^{lin}},  \psi_p \b{e_2} \rangle_{\b{-1,1}} &=0 \quad \forall p \in  \mathcal{V}_h^C \cup \mathcal{M}_h^C.\label{prop22}
\end{align}
Next, we subtract equations (\ref{prop22}) and (\ref{prop11}) from equation (\ref{4.2}) to get
\begin{align*}
\b{G_h(v)}&=\sum_{i=1}^{2}  \sum_{p \in (\mathcal{V}_h^o \cup \mathcal{M}_h^o) \setminus (\mathcal{V}_h^C \cup  \mathcal{M}_h^C)   }  \langle R^{lin}_i, (v_i-c_p(v_i)) \psi_p \rangle_{-1,1} + \sum_{p \in \mathcal{V}_h^C \cup \mathcal{M}_h^C}  \langle R^{lin}_2, (v_2-c_p(v_2)) \psi_p \rangle_{-1,1}  \nonumber  \\ & \hspace{0.3cm}+ \sum_{p \in \mathcal{V}_h^C \cup \mathcal{M}_h^C}  \langle R^{lin}_1, v_1 \psi_p \rangle_{-1,1} -\langle \tilde{\lambda}^{h}_1, v_1 \rangle_{-1,1}.
\end{align*}
Now, using the equation (\ref{3.23}) and \eqref{i0i}, we obtain the following equation
\begin{align*}
\b{G_h(v)}&=\sum_{i=1}^{2}  \sum_{p \in (\mathcal{V}_h^o \cup \mathcal{M}_h^o) \setminus (\mathcal{V}_h^C \cup  \mathcal{M}_h^C)   }  \langle R^{lin}_i, (v_i-c_p(v_i)) \psi_p \rangle_{-1,1} + \sum_{p \in \mathcal{V}_h^C \cup \mathcal{M}_h^C}  \langle R^{lin}_2, (v_2-c_p(v_2)) \psi_p \rangle_{-1,1}  \nonumber  \\ & \hspace{0.3cm}+ \sum_{p \in \mathcal{V}_h^C \cup \mathcal{M}_h^C}  \langle R^{lin}_1,(v_1-c_p(v_1)) \psi_p \rangle_{-1,1}.
\end{align*}
Finally, using equation (\ref{residual}),  H\"older's inequality and Lemma \ref{sign2}, we find
\begin{align*}
\b{G_h(v)}&=\sum_{i=1}^{2}  \sum_{p \in \mathcal{V}_h^o \cup \mathcal{M}_h^o} \Bigg(\int_{\omega_p}  r_i(\b{u^h}) (v_i-c_p(v_i)) \psi_p ~dx + \int_{\gamma_{p,I}} J^I_i(\b{u^h}) (v_i-c_p(v_i))\psi_p~ds \Bigg) \\ &\hspace{0.3cm}+\sum_{i=1}^{2}  \sum_{p \in \mathcal{V}_h^N \cup \mathcal{M}_h^N}  \int_{\gamma_{p,N}} J^N_i(\b{u^h}) (v_i-c_p(v_i))\psi_p~ds  - \sum_{p \in \mathcal{V}_h^C \cup \mathcal{M}_h^C}  \int_{\gamma_{p,C}} \hat{\sigma}_2(\b{u^h}) (v_2-c_p(v_2))\psi_p~ds  \nonumber  \\ &\hspace{0.3cm}- \sum_{p \in \mathcal{V}_h^C \cup \mathcal{M}_h^C}  \int_{\gamma_{p,C}} \hat{\sigma}_1(\b{u^h}) (v_1-c_p(v_1))\psi_p ~ds \nonumber \\ & \lesssim \Bigg( \sum_{i=1}^{2}  \sum_{p \in \mathcal{V}_h^o \cup \mathcal{M}_h^o}  h_p^2 \|r_i(\b{u^h})\|^2_{\b{L^2}(\omega_{p})}\Bigg)^{\frac{1}{2}} \Bigg( \sum_{i=1}^{2}  \sum_{p \in \mathcal{V}_h^o \cup \mathcal{M}_h^o}  \|\nabla v_i\|^2_{\b{L^2}(\omega_p)}\Bigg)^{\frac{1}{2}}
 \\ & \hspace{0.3cm}+ \Bigg( \sum_{i=1}^{2}  \sum_{p \in \mathcal{V}_h^o \cup \mathcal{M}_h^o}  h_p \|J^I_i(\b{u^h})\|^2_{\b{L^2}(\gamma_{p,I})}\Bigg)^{\frac{1}{2}} \Bigg( \sum_{i=1}^{2}  \sum_{p \in \mathcal{V}_h^o \cup \mathcal{M}_h^o}  \|\nabla v_i\|^2_{\b{L^2}(\omega_p)}\Bigg)^{\frac{1}{2}} \\ & \hspace{0.3cm}+ \Bigg( \sum_{i=1}^{2}  \sum_{p \in \mathcal{V}_h^N \cup \mathcal{M}_h^N}  h_p \|J^N_i(\b{u^h})\|^2_{\b{L^2}(\gamma_{p,N})}\Bigg)^{\frac{1}{2}} \Bigg( \sum_{i=1}^{2}  \sum_{p \in \mathcal{V}_h^N \cup \mathcal{M}_h^N} \|\nabla v_i\|^2_{\b{L^2}(\omega_p)}\Bigg)^{\frac{1}{2}} \\ & \hspace{0.3cm}+ \Bigg( \sum_{p \in \mathcal{V}_h^C \cup \mathcal{M}_h^C}  h_p \|\hat{\sigma}_2(\b{u^h})\|^2_{\b{L^2}(\gamma_{p,C})}\Bigg)^{\frac{1}{2}} \Bigg(  \sum_{p \in \mathcal{V}_h^C \cup \mathcal{M}_h^C}  \|\nabla v_2\|^2_{\b{L^2}(\omega_p)}\Bigg)^{\frac{1}{2}} \\ & \hspace{0.3cm}+ \Bigg( \sum_{p \in \mathcal{V}_h^C \cup \mathcal{M}_h^C}  h_p \|\hat{\sigma}_1(\b{u^h})\|^2_{\b{L^2}(\gamma_{p,C})}\Bigg)^{\frac{1}{2}} \Bigg( \sum_{p \in \mathcal{V}_h^C \cup \mathcal{M}_h^C}  \|\nabla v_1\|^2_{\b{L^2}(\omega_p)}\Bigg)^{\frac{1}{2}} \\ & \lesssim \Bigg( \sum_{i=1}^{5} \eta^2_i \Bigg)^{\frac{1}{2}} \|\b{v}\|_{\b{V}}.
\end{align*}
This completes the proof of this lemma.
\end{proof}
Next, we find an upper bound on the term $\b{\langle \b{\tilde{\lambda}^{h}}- \b{\lambda}, \b{u-u^h} \rangle_{-1,1}}$.
\begin{lemma} \label{lem:LM}
It holds that
\begin{align*}
\b{\langle \b{\tilde{\lambda}^{h}}- \b{\lambda}, \b{u-u^h} \rangle_{-1,1}}& \lesssim \frac{1}{2} \|\lambda_1 -\tilde{\lambda}^{h}_1\|^2_{H^{-1}(\Omega)}+ \eta_6^2 +\eta_7^2.    
\end{align*}
\end{lemma}
\begin{proof}
Let $\tilde{u}^{h}_1:= \min \{u^h_1|_{\Gamma_C}, 0 \} \in H^{\frac{1}{2}}(\Gamma_C)$. Let $\tilde{w}$ be the harmonic extension 
 of $w= u^h_1-\tilde{u}^{h}_1 \in H^{\frac{1}{2}}(\Gamma_C) $ such that $\|\tilde{w}\|_{H^1(\Omega)} \lesssim \|w\|_{H^{\frac{1}{2}}(\Gamma_C)}$\cite{Ste:2008:book}. A use of \eqref{3.7} and Lemma \ref{sign1} yields
\begin{align} \label{4.5}
\b{\langle \b{\tilde{\lambda}^{h}}- \b{\lambda}, \b{u-u^h} \rangle_{-1,1}}&= \b{\langle \b{\tilde{\lambda}^{h}},\b{u-u^h} \rangle_{-1,1}} + \b{\langle   \b{\lambda}, \b{u^h-u} \rangle_{-1,1}} \nonumber\\
&=\langle {\tilde{\lambda}^{h}_1},{u_1-u^h_1} \rangle_{-1,1} + \langle   {\lambda_1}, {u^h_1-u_1} \rangle_{-1,1}.
\end{align}
Employing the relation \eqref{3.7}, we deal with the second term on the right hand side of the last equation as follows
\begin{align*}
\langle  \lambda_1, u^h_1-u_1 \rangle_{-1,1} 
&= \langle  \lambda_1, u^h_1-\tilde{u}^{h}_1 \rangle_{-1,1} + \langle \lambda_1, \tilde{u}^{h}_1-u_1 \rangle_{-1,1}.
\end{align*}
By definition of $\tilde{u}^{h}_1$, we have $\tilde{u}^{h}_1 \leq 0$ on $\Gamma_C$,  therefore by Lemma \ref{CFD1} we have
\begin{align*}
\langle  \lambda_1, u^h_1-u_1 \rangle_{-1,1} & \leq \langle  \lambda_1, u^h_1-\tilde{u}^{h}_1 \rangle_{-1,1} \\ & =\langle \lambda_1 -\tilde{\lambda}^{h}_1 , u^h_1-\tilde{u}^{h}_1 \rangle_{-1,1}  + \langle \tilde{\lambda}^{h}_1 , u^h_1-\tilde{u}^{h}_1 \rangle_{-1,1}.
\end{align*} 
Using the Young's inequality and the stability estimate for the harmonic extension, we find
\begin{align*}
\langle  \lambda_1, u^h_1-u_1 \rangle_{-1,1} & \lesssim \frac{1}{2} \|\lambda_1 -\tilde{\lambda}^{h}_1\|^2_{H^{-1}(\Omega)} + \frac{1}{2} \|u^h_1-\tilde{u}^{h}_1\|^2_{H^{\frac{1}{2}}(\Gamma_C)}  + \langle \tilde{\lambda}^{h}_1 , u^h_1-\tilde{u}^{h}_1 \rangle_{-1,1}.
\end{align*}
Thus, using equation (\ref{4.5}), we get
\begin{align} \label{4.6}
\langle \tilde{\lambda}^{h}_1- \lambda_1, u_1-u^h_1 \rangle_{-1,1} & \lesssim \frac{1}{2} \|\lambda_1 -\tilde{\lambda}^{h}_1\|^2_{H^{-1}(\Omega)} + \frac{1}{2} \|u^h_1-\tilde{u}^{h}_1\|^2_{H^{\frac{1}{2}}(\Gamma_C)}  + \langle \tilde{\lambda}^{h}_1 , u^h_1-\tilde{u}^{h}_1+u_1-u^h_1 \rangle_{-1,1}.
\end{align}
Note that on $\Gamma_C$, we have
\begin{align*}
(u^h_1-\tilde{u}^{h}_1)+u_1-u^h_1 &= (u^h_1)^++u_1-u^h_1 \\ & \leq (u^h_1)^+ -u^h_1 \\ &= (-u^h_1)^+.
\end{align*}
With this realization and in view of Lemma \ref{sign2}, we handle the third term on the right hand side of equation (\ref{4.6}) as follows
\begin{align*}
\langle \tilde{\lambda}^{h}_1 , u^h_1-\tilde{u}^{h}_1-u^h_1+u_1 \rangle_{-1,1} &= \sum_{p \in \mathcal{V}_h^C \cup \mathcal{M}_h^C} \langle \tilde{\lambda}^{h}_1 , (u^h_1-\tilde{u}^{h}_1-u^h_1+u_1) \psi_p \rangle_{-1,1}  \\ &\leq \sum_{p \in \mathcal{V}_h^C \cup \mathcal{M}_h^C} \langle \tilde{\lambda}^{h}_1 , (-u^h_1)^+\psi_p \rangle_{-1,1}. 
\end{align*}
Using the Lemma \ref{sign}, equations \eqref{3.23}, \eqref{i0i} together with the remark (\ref{rem}), we have 
\begin{align}\label{555}
\langle \tilde{\lambda}^{h}_1 , u^h_1-\tilde{u}^{h}_1-u^h_1+u_1 \rangle_{-1,1} &\leq \sum_{p \in \mathcal{N}^{FC}_h} s_p^1 c_p((-u^h_1)^+)\int_{\gamma_{p,C}}\phi_p~ds + \sum_{p \in \mathcal{N}^{SC}_h} s_p^1 c_p((-u^h_1)^+)\int_{\gamma_{p,C}}\phi_p~ds.
\end{align}
Since for any full contact node $p$, i.e., ${p \in \mathcal{N}^{FC}_h}$ we have $u^h_1=0$ on $\gamma_{p,C}$, thus  $(-u^h_1)^+ =0$. Therefore, the equation \eqref{555} reduces to
\begin{align*}
\langle \tilde{\lambda}^{h}_1 , u^h_1-\tilde{u}^{h}_1-u^h_1+u_1 \rangle_{-1,1} &\leq  \sum_{p \in \mathcal{N}^{SC}_h} s_p^1 c_p((-u^h_1)^+)\int_{\gamma_{p,C}}\phi_p~ds\\
&= \sum_{p \in \mathcal{N}^{SC}_h} s_p^1 \Bigg(  \int_{\gamma_{p,C}}\phi_p \frac{ \int_{\tilde{\gamma}_{p,C}} (-u^h_1)^+\phi_p ~ds}{\int_{\tilde{\gamma}_{p,C} \phi_p~ds}} \Bigg),
\end{align*}
where we have used the definition of constants $c_p((-u^h_1)^+)$ on semi contact nodes in the last step.

We observe that $\dfrac{\int_{\gamma_{p,C}}\phi_p~ds}{\int_{\tilde{\gamma}_{p,C} }\phi_p~ds}$ is a computable constant independent of $h$ where $\tilde{\gamma}_{p,C}$ is a strict subset of $\gamma_{p,C}$.
Therefore,
\begin{align*}
\langle \tilde{\lambda}^{h}_1 , u^h_1-\tilde{u}^{h}_1-u^h_1+u_1 \rangle_{-1,1} \lesssim \sum_{p \in \mathcal{N}^{SC}_h} s_p^1 d_p,
\end{align*}
where $d_p = \int_{\tilde{\gamma}_{p,C}} (-u^h_1)^+\phi_p~ds.$
Finally, we obtain
\begin{align} \label{4.8}
\langle  \lambda_1-\tilde{\lambda}^{h}_1, u^h_1-u_1 \rangle_{-1,1} & \lesssim 
\frac{1}{2} \|\lambda_1 -\tilde{\lambda}^{h}_1\|^2_{H^{-1}(\Omega)} + \eta_6^2 + \eta_7^2.
\end{align}
This completes the proof.
\end{proof}
In view of the Lemmas \ref{lem:Rel}, \ref{lem:Gh} and \ref{lem:LM}, we have the desired reliability estimate of the error estimator $\eta_h$.
\begin{remark}
The error estimator $\eta_h$ is comparable with the error estimator derived in \cite{Krause:2015:apost_Sig} for the linear conforming finite element method for the Signorini problem.

\end{remark}
\subsection{Efficiency of the error estimator}
This subsection is devoted to establish the efficiency of the estimator $\eta_h$. We will accomplish it using standard bubble function arguments \cite{Verfurth:1994:Adaptive}.
We would like to remark here that the efficiency of the estimator terms $\eta_6$ and $\eta_7$ involving positive part of $\b{u^h}$ is less clear theoretically due to quadratic nature of discrete solution and this will be pursued in future.
\begin{lemma} Let $\b{u} \in \b{V}$ be the solution of continuous problem \eqref{1.2} and $\b{u^h} \in \b{V^h}$ be the solution of discrete problem \eqref{2.1}. Then, the lower bound on the estimators for $k= 1,2$ with $p \in \mathcal{V}_h^0 \cup \mathcal{M}_h^0$,  for $k=3$ with $p\in \mathcal{V}_h^N \cup \mathcal{M}_h^N$, for $k=4,5$ with $p\in \mathcal{V}_h^C \cup \mathcal{M}_h^C$ follow as:
  \begin{align*}
 \eta_{k,p} \lesssim \|\b{u-u^h}\|_{\b{H^1}(\omega_p)} +\|\b{\lambda}-\b{\tilde{\lambda}_h}\|_{\b{H^{-1}}(\omega_p)}+ Osc(\b{f})^2+Osc(\b{g})^2,
   \end{align*}
where the oscillation terms are defined as
    \begin{align*}
    Osc(\b{f})^2&= \sum_{T\in \omega_p} h^2_T\|\b{f-\bar{f}}\|^2_{\b{L^2}(T)},\\
    Osc(\b{g})^2&=\sum_{e\in \gamma_{p,N}} h_e\|\b{g-\bar{g}}
    \|^2_{\b{L^2}(e)},
    \end{align*}
with $\b{\bar{v}}$ representing the $\b{L^2}$ projection of $\b{v}$ onto the space of piece-wise constant functions.
\end{lemma}

\begin{proof}
\textbf{($i$)~(Local bound for $\eta_1$)} To this end, we choose an arbitrary triangle $T \in \omega_p$. Let $b_{T} \in P_3(T)$ be the interior bubble function which is zero on $\partial T$ and admits unit value at the barycenter of $T$. Set $\b{\beta_T} = b_{T} (\b{\bar{f}}+\b{div}\b{\sigma}(\b{u^h}))$ on $T$, where $\b{\bar{f}}$ is the function with the piece wise constant approximation of $\b{f}$ with the components $\bar{f_i}$. Extend $\b{\beta_T}$ to $\Omega$ by defining it to be zero on $\Omega \backslash \overline{T}$, call it $\b{\beta}$.  It is evident that $\b{\beta} \in [H^1_0(\Omega)]^2$. Using the equivalence of norms in finite dimensional normed spaces on a reference triangle and scaling arguments \cite{Ciarlet:1978:FEM}, we find
    \begin{align}\label{5.1}
     \|\b{\bar{f}}+\b{div}\b{\sigma}(\b{u^h})\|^2_{\b{L^2}(T)} &\lesssim \int_{T} \b{\beta}\cdot (\b{\bar{f}+ \b{div}\b{\sigma}(\b{u^h})})~dx \nonumber\\
     &= \int_{T} \b{\beta}\cdot (\b{\bar{f} - f})~dx+ \int_{T} \b{r(u^h)}\cdot \b{\beta}~dx. 
    \end{align}
Using integration by parts, \eqref{CFD} together with the definition of functional $\b{G^h}$ in \eqref{5.1}, we find
\begin{align*}
\|\b{\bar{f}}+\b{div}\b{\sigma}(\b{u^h})\|^2_{\b{L^2}(T)}
&\lesssim \int_{T} \b{\beta}\cdot (\b{\bar{f} - f})~dx+ \b{\langle\b{G_h},\b{\beta}\rangle_{-1,1,T}}.
\end{align*}
Henceforth using Cauchy Schwartz inequality and standard inverse estimates, we obtain
\begin{align*}
\|\b{\bar{f}}+\b{div}\b{\sigma}(\b{u^h})\|^2_{\b{L^2}(T)} &\lesssim  \big(\|\b{\bar{f}-f}\|_{\b{L^2}(T)} + h_T^{-1}\|\b{G_h}\|_{{\b{H^{-1}}(T)}}\big)\b{\|\b{\beta}}\|_{\b{L^2}(T)} \\
&\lesssim  \big(\|\b{\bar{f}-f}\|_{\b{L^2}(T)} + h_T^{-1}\|\b{G_h}\|_{{\b{H^{-1}}(T)}}\big)\|\b{\bar{f}}+\b{div}\b{\sigma}(\b{u^h})\|_{\b{L^2}(T)}.
\end{align*}
Thus, we find
\begin{align*}
  h_T^2\|\b{\bar{f}}+\b{div}\b{\sigma}(\b{u^h})\|^2_{\b{L^2}(T)}\lesssim  h_T^2\|\b{\bar{f}-f}\|^2_{\b{L^2}(T)} +\|\b{G_h}\|^2_{{\b{H^{-1}}(T)}}.
\end{align*}
Using the definition of Galerkin functional in the last equation, it follows that 
\begin{align*}
  h_T^2\|\b{\bar{f}}+\b{div}\b{\sigma}(\b{u^h})\|^2_{\b{L^2}(T)}\lesssim  h_T^2\|\b{\bar{f}-f}\|^2_{\b{L^2}(T)} +|\b{u-u^h}|^2_{\b{H^1}(T)} + \|\b{\lambda} -\b{\tilde{\lambda}^{h}}\|^2_{\b{H^{-1}}(T)}.
\end{align*}
Summing the above equation, over all the triangles in $\omega_p$ and using triangle inequality, we get 
\begin{align*}
    \eta_{1,p}^2 \lesssim \sum_{T \in \omega_p}h_T^2\|\b{\bar{f}-f}\|^2_{\b{L^2}(T)} +|\b{u-u^h}|^2_{\b{H^1}(\omega_p)} + \|\b{\lambda} -\b{\tilde{\lambda}^{h}}\|^2_{\b{H^{-1}}(\omega_p)}.
\end{align*}
\par
\textbf{($ii$)~(Local bound for $\eta_2$)} In order to prove this, let $e \in \gamma_{p,I}$ be an arbitrary interior edge sharing the elements $T^+$ and $T^-$ and let $\b{n_e}$ be the outward unit vector normal vector to $e$ heading from the triangle $T^-$ to $T^+$. To this end, we will construct an edge bubble function and exploit its properties as follows: Define $b_{e} \in P_4(T^- \cup T^+)$ be the polynomial such that it takes value 1 at the midpoint  of $e$ and zero on the boundary of polygon $T^- \cup T^+$. Further, we define $\b{\zeta} \in [P_1(T^{-} \cup T^{+})]^2$ to be the polynomial such that $\b{\zeta}=\sjump{\b{\sigma}(\b{u^h})}$ on edge e. Set $\b{\beta}= b_{e} \b{\zeta} $ on $T^- \cup T^+$ and zero outside the polygon $T^- \cup T^+$ yielding $\b{\beta} \in [H^1_0(\Omega)]^2$. A use of equivalence of norms on a reference element in finite dimensional spaces together with scaling arguments yields
    \begin{align}\label{5.3}
    \|\b{\zeta}\|^2_{\b{L^2}(e)} &\lesssim \int_e \b{\beta} \cdot \b{\zeta}~ds.
    \end{align}
Using integration by parts together with the definition of $\b{G^h}$ we find
\begin{align*}
    \int_{e}\b{\zeta} \cdot \b{\beta}~ds &=\int_{T^+\cup T^-} \b{\sigma}(\b{u^h}):\b{\epsilon}(\b{\beta})~dx + \int_{T^+\cup T^-}\b{div} \b{\sigma}(\b{u^h})\cdot\b{\beta}~dx\\ 
&= \int_{T^+\cup T^-} \b{\sigma}(\b{u^h}):\b{\epsilon}(\b{\beta})~dx + \int_{T^+\cup T^-}\b{r(u^h)}\cdot\b{\beta}~dx -\int_{T^+\cup T^-}\b{f}\cdot \b{\beta}~dx \\
&= \b{\langle \b{G_h},\b{\beta}\rangle}_{\b{-1,1},{T^+\cup T^-}} + \int_{T^+\cup T^-}\b{r(u^h)}\cdot\b{\beta}~dx.
\end{align*}
Further a use of Cauchy Schwartz inequality and standard inverse estimate yields
\begin{align}\label{5.4}
\int_{e}\b{\zeta} \cdot \b{\beta}~ds &\lesssim {\sum_{T \in {T^+\cup T^-}}}\big(h^{-1}_T\|\b{G_h}\|_{\b{H^{-1}}(T)} + \|\b{r(u^h)}\|_{\b{L^2}(T)}\big)\|\b{\beta}\|_{\b{L^2}(T)}  \nonumber\\  
&\lesssim {\sum_{T \in {T^+\cup T^-}}}\big(h^{-1}_T\|\b{G_h}\|_{\b{H^{-1}}(T)} + \|\b{r(u^h)}\|_{L^2(T)}\big)h^{1/2}_e \|\b{\zeta}\|_{\b{L^2}(e)}.
\end{align} 
Thus, combining $\eqref{5.3}$ and $\eqref{5.4}$, we find 
\begin{align}\label{5.5}
    h_e^{1/2}\|\sjump{\b{\sigma}(\b{u^h})}\|_{\b{L^2}(e)} &\lesssim \sum_{T \in {T^+\cup T^-}} \Big( \|\b{G_h}\|_{\b{H^{-1}}(T)}+ h_T \|\b{r(u^h)}\|_{\b{L^2}(T)} \Big).
    \end{align}
In view of the definition of Galerkin functional $\b{G_h}$ together with the estimate in $\b{(i)}$, we find 
\begin{align*}
    h_e\|\sjump{\b{\sigma}(\b{u^h})}\|^2_{\b{L^2}(e)} &\lesssim \sum_{T \in {T^+\cup T^-}} \Big(h^2_T\|\b{\bar{f}-f}\|^2_{\b{L^2}(T)} +|\b{u-u^h}|^2_{\b{H^1}(T)} + \|\b{\lambda} -\b{\tilde{\lambda}^{h}}\|^2_{\b{H^{-1}}(T)} \Big).
    \end{align*}
Summing over all edges in $\gamma_{p,I}$, we obtain a local bound as
\begin{align*}
     h_s\|\sjump{\b{\sigma}(\b{u^h})}\|^2_{\b{L^2}(\gamma_{p,I})} &\lesssim \sum_{T \in \omega_p}h^2_T\|\b{\bar{f}-f}\|^2_{\b{L^2}(T)} +|\b{u-u^h}|^2_{\b{H^1}(\omega_p)} + \|\b{\lambda} -\b{\tilde{\lambda}^{h}}\|^2_{\b{H^{-1}}(\omega_p)},
\end{align*}
which is the desired estimate.

\par
{\textbf{($iii$)~(Local bound for $\eta_3$)}}: The bound for the estimator term $\eta_3$ can be proved analogously to the estimator $\eta_2$. 
\par
\textbf{($iv$)~(Local bound for $\eta_4$)} In order to estimate $\eta_4$, let $e$ be an arbitrary edge of $\gamma_{p,C}$ such that the edge $e$ is the part of triangle $T$. Let $n_e$ denotes the outward unit normal vector to the edge $e$ which in this article is assumed to be (1,0). Next, we define a bubble function $b_{e} \in P_2(T)$  which vanishes on $\partial {T} \setminus e$ and takes value 1 at the midpoint of $e$. Let $\b{\zeta} \in [P_1(T)]^2$ be a polynomial such that its normal component is zero. i.e. $\zeta_n$~=~$\zeta_{1}=0$ and its tangential component i.e. $\b{\zeta_{\tau}}=(0,~\zeta_{2})=\hat{\b{\sigma}}_2(\b{u^h})$. Define $\b{\beta} = b_{e}\b{\zeta}$ on $T$ with its trivial extension to outside of $T$ which belongs to $\b{V}$. Further, this implies  $\b{\beta}~=\b{\beta_{\tau}}= (0,~\beta_2)$. A use of scaling arguments and equivalence of norms on finite dimensional spaces guarantees that
\begin{align}\label{5.6}
\|\b{\hat{\b{\sigma}}_2(\b{u^h})}\|^2_{\b{L^2}(e)} \lesssim \int_e \b{\hat{\b{\sigma}}_2(\b{u^h})} \cdot \b{\beta}~ds.
\end{align}
Using integration by parts and \eqref{CFD}, we have
\begin{align*}
\int_e \b{\hat{\b{\sigma}}_2(\b{u^h})} \cdot \b{\beta}~ds &= \int_{T} \b{\sigma}(\b{u^h}):\b{\epsilon}(\b{\beta})~dx + \int_{T} \b{div} \b{\sigma}(\b{u^h})\cdot \b{\beta}~dx\\
&= \int_{T} \b{\sigma}(\b{u^h}):\b{\epsilon}(\b{\beta})~dx + \int_{T} \b{r(u^h)}\cdot \b{\beta}~dx - \int_{T} \b{f}\cdot \b{\beta}~dx \\
&= \int_{T} \b{\sigma}(\b{u^h}):\b{\epsilon}(\b{\beta})~dx + \int_{T} \b{r(u^h)}\cdot \b{\beta}~dx - a(\b{u}, \b{\beta}) - \b{\langle \b{\lambda}, \b{\beta} \rangle_{-1,1}}. 
\end{align*}
In view of the definition (2.7), we have $\b{\langle \b{\lambda},~ \b{\beta} \rangle_{-1,1}} = 0 $. Therefore, the continuity of bilinear form $a(\cdot, \cdot)$ together with Cauchy Schwartz inequality and standard inverse estimate yields
\begin{align}\label{5.7}
\int_e \b{\hat{\b{\sigma}}_2(\b{u^h})} \cdot \b{\beta}~ds &= a(\b{u-u^h},\b{\beta}) + \int_{T} \b{r(u^h)}\cdot \b{\beta}~dx \nonumber\\
&\lesssim \big(h_T^{-1}|\b{u-u^h}|_{\b{H^1}(T)}+\|\b{r(u^h)}\|_{\b{L^2}(T)}\big)\|\b{\beta}\|_{\b{L^2}(T)}  \nonumber\\ 
&\lesssim \big(h_T^{-1}|\b{u-u^h}|_{\b{H^1}(T)}+\|\b{r(u^h)}\|_{\b{L^2}(T)}\big)h^{1/2}_e \|\b{\zeta}\|_{\b{L^2}(e)}.
\end{align}
Combining $\eqref{5.6}$ and $\eqref{5.7}$ together with the estimate in $\b{(i)}$, we obtain 
\begin{align*}
 h_e\|\b{\hat{\b{\sigma}}_2(\b{u^h})}\|^2_{\b{L^2}(e)} &\lesssim h^2_T\|\b{\bar{f}-f}\|^2_{\b{L^2}(T)} +|\b{u-u^h}|^2_{\b{H^1}(T)} + \|\b{\lambda} -\b{\tilde{\lambda}^{h}}\|^2_{\b{H^{-1}}(T)}.
\end{align*}
Further, summing over all the $e \in \gamma_{p,C}$, we find
\begin{align*}
    \eta^2_{4,p} &\lesssim \sum_{T \in \omega_p}h^2_T\|\b{\bar{f}-f}\|^2_{\b{L^2}(T)} +|\b{u-u^h}|^2_{\b{H^1}(\omega_p)} + \|\b{\lambda} -\b{\tilde{\lambda}^{h}}\|^2_{\b{H^{-1}}(\omega_p)}.   
\end{align*}
\par
\noindent	
\textbf{($v$)~(Local bound for $\eta_5$)}
Let $p\in \mathcal{V}^C_h \cup \mathcal{M}^C_h$ be arbitrary. Let $e \in \gamma_{p,C}$ be arbitrary. On the similar lines as in the proof of $(iii)$ we construct an edge bubble function $\zeta_e \in P_2(T)$. Using the definition of Galerkin functional $\b{G^h}$, we have
\begin{align*}
 \b{G_h}(\zeta_e \b{e_1}) &=\sum_{p \in (\mathcal{V}_h^o \cup \mathcal{M}_h^o) \setminus (\mathcal{V}_h^C \cup  \mathcal{M}_h^C)   }  \langle R^{lin}_1, \zeta_e \psi_p \rangle_{-1,1} + \sum_{p \in \mathcal{V}_h^C \cup \mathcal{M}_h^C}  \langle R^{lin}_1, \zeta_e \psi_p \rangle_{-1,1}\\ &- \sum_{p \in \mathcal{V}_h^C \cup \mathcal{M}_h^C} \langle \tilde{\lambda}^{h}_1, \zeta_e \psi_p \rangle_{-1,1}   \\
 &=\sum_{p \in (\mathcal{V}_h^o \cup \mathcal{M}_h^o)} \int_{\omega_p}r_1\zeta_e\psi_p ~dx-
 \sum_{p \in (\mathcal{V}_h^C \cup \mathcal{M}_h^C)} \int_{\gamma_{p,C}}\hat{\b{\sigma_1}}(\b{u^h})\zeta_e\psi_p ~ds\\&- \sum_{p \in \mathcal{V}_h^C \cup \mathcal{M}_h^C} \langle \tilde{\lambda}^{h}_1, \zeta_e \psi_p \rangle_{-1,1}.
\end{align*}
Now, if $p$ is a non actual contact node, then using Lemma \ref{sign} it holds that $\b{\tilde{\lambda}^{h}}=0$, thus we can proceed similar to the proof of lower bound of $\eta_4$.
We  consider the case when the node $p$ is a full contact node or semi contact node then the above equation reduces to
\begin{align*}
 \b{G_h}(\zeta_e \b{e_1}) &=\sum_{p \in (\mathcal{V}_h^o \cup \mathcal{M}_h^o)} \int_{\omega_p}r_1\zeta_e\psi_p -
 \sum_{p \in (\mathcal{V}_h^C \cup \mathcal{M}_h^C)} \int_{\gamma_{p,C}}\hat{\b{\sigma_1}}(\b{u^h})\zeta_e\psi_p \\&- \sum_{p \in \mathcal{N}_h^{FC} \cup \mathcal{N}_h^{SC}} s_p^1 c_p(\zeta_e) \int_{\gamma_{p,C}}\phi_p~ds.
\end{align*}
In order to get rid of the last term, we construct a suitable function $\theta_e$ such that $c_p(\theta_e)=0$. To this end, we will exploit the definition of $c_p(\cdot)$ which depends on $\tilde{\gamma}_{p,C}$. If $p$ is an interior vertex of $\mathcal{V}_h^C$, then $\gamma_{p,C}$ consists of two intervals. In that case we set $\tilde{\gamma}_{p,C}$ as inner third of $\gamma_{p,C}$ containing p. In contrast to this if $p$ is midpoint in $\mathcal{M}_h^C$, then $\gamma_{p,C}$ consist of one interval. Let $e_i$ be the sides of subgrid containing $p_i$ where $p_i \in \mathcal{V}^C_h$ and $e_M$ be the part of the subgrid containing the midpoint of $[p_i,~p_{i+1}]$ (see Figure 5.1). Now, we will use the above construction to define the function $\theta_e$ in the following way
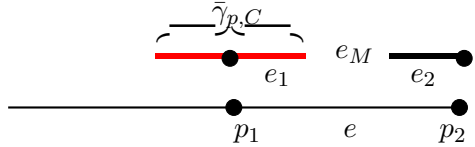
\begin{figure}[ht!]
\centering
\hspace{0cm}$\bar{\gamma}_{p,C}$ \\[1ex]
\begin{tikzpicture}
\hspace{-1.6cm}$\overbrace{\color{red}\rule{2cm}{2pt}}$  
\filldraw[black] (-1,0) circle(3pt);
\hspace{0.4cm}$e_M$\hspace{0.2cm}\rule{1cm}{2pt} 
\filldraw[black] (0,0) circle(3pt);	
\end{tikzpicture}\\[-1.5ex]
\hspace{3cm}$e_1$ \hspace{1.5cm}$e_2$ \\
\begin{tikzpicture}
\draw[black,thick](-3,0) -- (3,0);
\filldraw[black] (0,0) circle(3pt);
\filldraw[black] (3,0) circle(3pt);	
\end{tikzpicture} \\[-1ex]
\hspace{3cm}$p_1$ \hspace{1cm}$e$ \hspace{1cm}$p_2$\\
\caption{Subgrid of $\gamma_{p,C}$} \label{patch1}
\end{figure}





\begin{align}
\theta_e= \sum_{i=1}^{2} \alpha_i \psi_i + \alpha_M \psi_M,
\end{align} where $\psi_M$ and $\psi_i$ are the edge bubble functions corresponding to $e_M$ and $e_i$ respectively. The coefficients $\alpha_i$ and $\alpha_M$ are determined such that the following holds
\begin{enumerate}
\item[1.] $\int_e 1 = \underset{{p \in \mathcal{N}^{FC}_h \cup \mathcal{N}^{SC}_h}}{\sum} \int_e \hat{\sigma_1}(\b{u^h})\theta_e \phi_p$.
\item[2.] $\int_{e_i} \hat{\sigma_1}(\b{u^h})\theta_e \phi_{p_i}=0 \quad \forall$ semi contact and full contact nodes lying on edge $e$.
\end{enumerate}
Thus, we have $c_p(\hat{\sigma}_1(\b{u^h})\theta_e)=0$. Further, using the equivalence of norms in finite dimensional spaces, H\"older's inequality and the construction of $\theta_e$, we have
\begin{align*}
      \|\hat{\sigma_1} (\b{u_h})\|^2_{\b{L}^{2}(e)} & \lesssim \int_e \hat{\sigma_1} (\b{u_h}) \hat{\sigma_1} (\b{u_h}) \theta_e \psi_p~ds \\ &\lesssim \Big(-\b{\langle \b{G}_h, \hat{\sigma_1} (\b{u^h}) \theta_e \b{e}_1 \rangle}_{\b{-1,1},\omega_p} + \int_{\omega_p} \b{r}_1(\b{u}^h) \hat{\sigma_1} (\b{u^h}) \theta_e \psi_p~dx\Big) \notag\\
      & \lesssim \Big( \|\b{G}_h\|_{\b{H^{-1}}{(\omega_p)}} \|\hat{\sigma_1} (\b{u^h}) \theta_e \|_{\b{H}^{1}(\omega_p)} + \|\b{r}(\b{u}^h)\|_{\b{L}^{2}(\omega_p)} \|\hat{\sigma_1} (\b{u^h}) \theta_e \|_{\b{L}^{2}(\omega_p)}\Big) \notag \\
      & \lesssim \Big( \|\b{G}_h\|_{\b{H^{-1}}{(\omega_p)}} h_e^{-1}\|\hat{\sigma_1} (\b{u^h}) \theta_e \|_{\b{L}^{2}(\omega_p)} + \|\b{r}(\b{u}^h)\|_{\b{L}^{2}(\omega_p)} \|\hat{\sigma_1} (\b{u^h}) \theta_e \|_{\b{L}^{2}(\omega_p)}\Big) \notag \\
      & \lesssim \Big( \|\b{G}_h\|_{\b{H^{-1}}{(\omega_p)}} + h_e\|\b{r}(\b{u}^h)\|_{\b{L}^{2}(\omega_p)}  \Big)h_e^{-\frac{1}{2}}\|\hat{\sigma_1} (\b{u^h})\|_{\b{L}^{2}(e)} \notag 
\end{align*}
Thus, 
\begin{align*}
    h_e\|\hat{\sigma_1} (\b{u_h})\|^2_{\b{L}^{2}(e)} & \lesssim \Big( \|\b{G}_h\|^2_{\b{H^{-1}}(\omega_p)} + {h_T}^2\|\b{r}(\b{u}^h)\|^2_{\b{L}^{2}(\omega_p)}  \Big).
\end{align*}
We conclude the proof using the upper bound of $\b{G_h}$ and the estimate in $(i)$.

\end{proof}   
\section{Numerical Results}\label{sec:NumResults}
\par
\noindent
The aim of the given section is to numerically illustrate the theoretical findings derived in section \ref{sec:Discrete}
 and section \ref{sec:Apoestriori}, respectively. Therein, the numerical experiments are performed on two model problems using MATLAB(version R2020b). The first model problem is constructed in such a way that the exact solution $\b{u}$ is a priori known. Henceforth, the exact error is computed and the results are compared with the convergence of a posteriori error estimator $\eta_h$. In second model problem, the exact solution is unknown and we focus on the convergence of the error estimator $\eta_h$ therein.
The discrete variational inequality is solved using the primal dual active set strategy \cite{Wohlmuth:2005:PrimalDual}.
We carried out these tests on adaptive mesh for which we will make use of the following paradigm
\begin{center}
    \textbf{SOLVE} $\longrightarrow$ \textbf{ESTIMATE} $\longrightarrow$ \textbf{MARK} $\longrightarrow$ \textbf{REFINE}
\end{center}
The step \textbf{SOLVE} comprises of computing the discrete solution $\b{u^h}$ by solving the discrete variational inequality with the help of primal-dual active set strategy . Thereafter, in the next step \textbf{ESTIMATE}, the error estimator $\eta_h$ discussed in section 4 is computed element wise and further making the use of D\"orfler's marking strategy \cite{Dorfler:1996:Afem} with the parameter $\theta= 0.4$, we mark the elements of the triangulation followed by that in the step \textbf{REFINE} the marked elements are refined using the newest vertex bisection algorithm to obtain the new mesh and the algorithm is repeated. Note that when $\Gamma_C$ lies on the $x$-axis, we have $\boldsymbol{n}=(0,-1)$ on $\Gamma_C$. Therefore in this case, the estimator given in equation \eqref{esti} will have modified estimator contributions $\eta_{6,p}=(s^1_pd_p)^{\frac{1}{2}}$, where $d_p= \int_{\tilde{\gamma}_{p,C}} (u^h_2)^+\phi_p~ds$ and 
$\eta_7 = ~\|(-u^h_2)^{+}\|_{H^{\frac{1}{2}}(\Gamma_C)}$.
\par
For the given examples, the Lame's parameter $\mu$ and $\chi$ are computed as follows
\begin{align*}
    \mu = \dfrac{E}{2(1+\nu)},~~~~~~~~\chi=\dfrac{E\nu}{(1-2\nu)(1+\nu)}
\end{align*}
where $E$ and $\nu$ represents the Young's modulus and Poisson ratio \cite{KO:1988:CPBook}, respectively.
\par
\begin{figure*}[!ht]
	\begin{minipage}{8 cm}
	\includegraphics[width=1.1\textwidth]{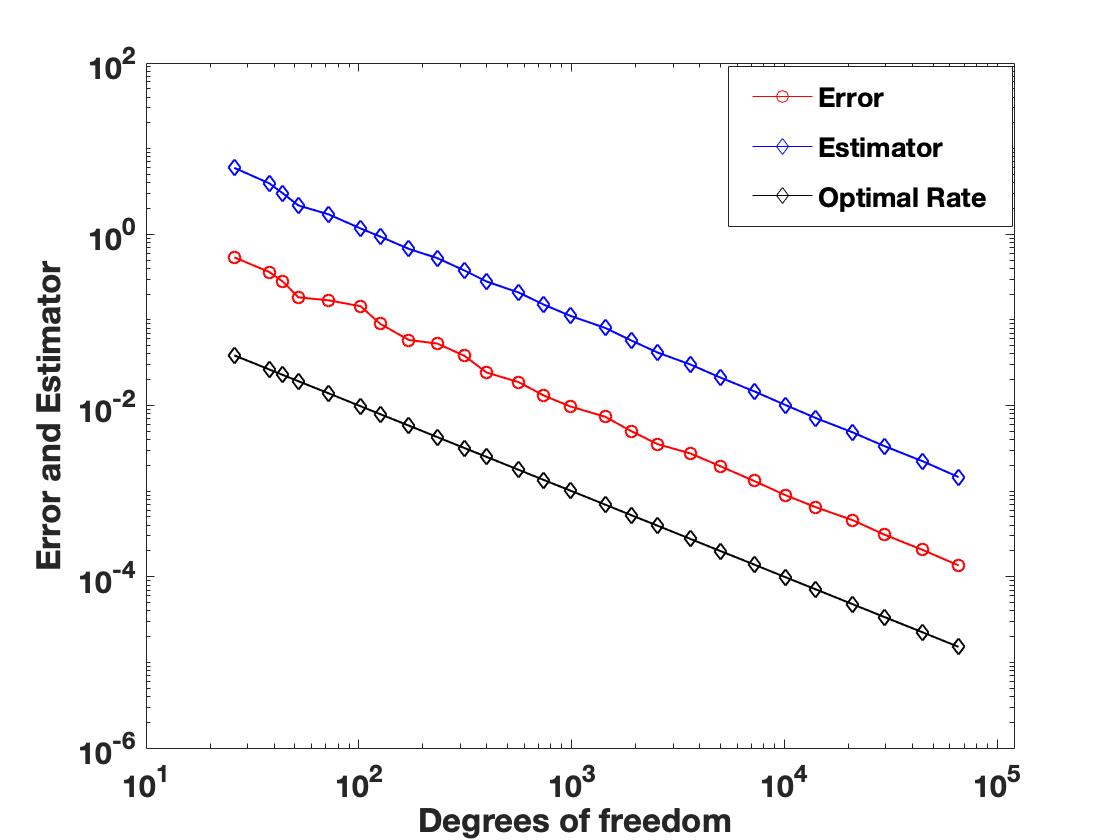}
	\end{minipage}
	\hspace{0.2cm}
	\begin{minipage}{8 cm}		
	\includegraphics[width=1.1\textwidth]{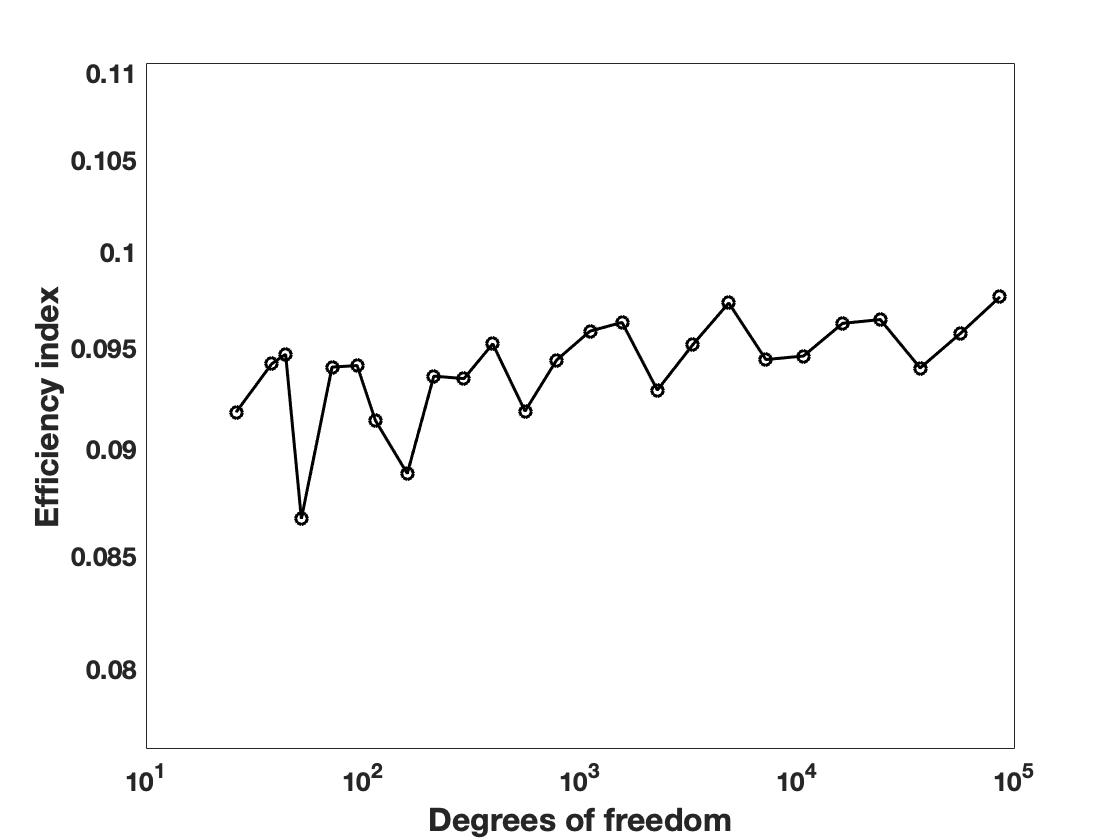}
	\end{minipage}
\caption{Convergence of Error, Estimator and Efficiency Index for Example 6.1.}
	\label{3}
\end{figure*}
\par
\begin{figure*}[!ht]
	\includegraphics[width=0.6\textwidth]{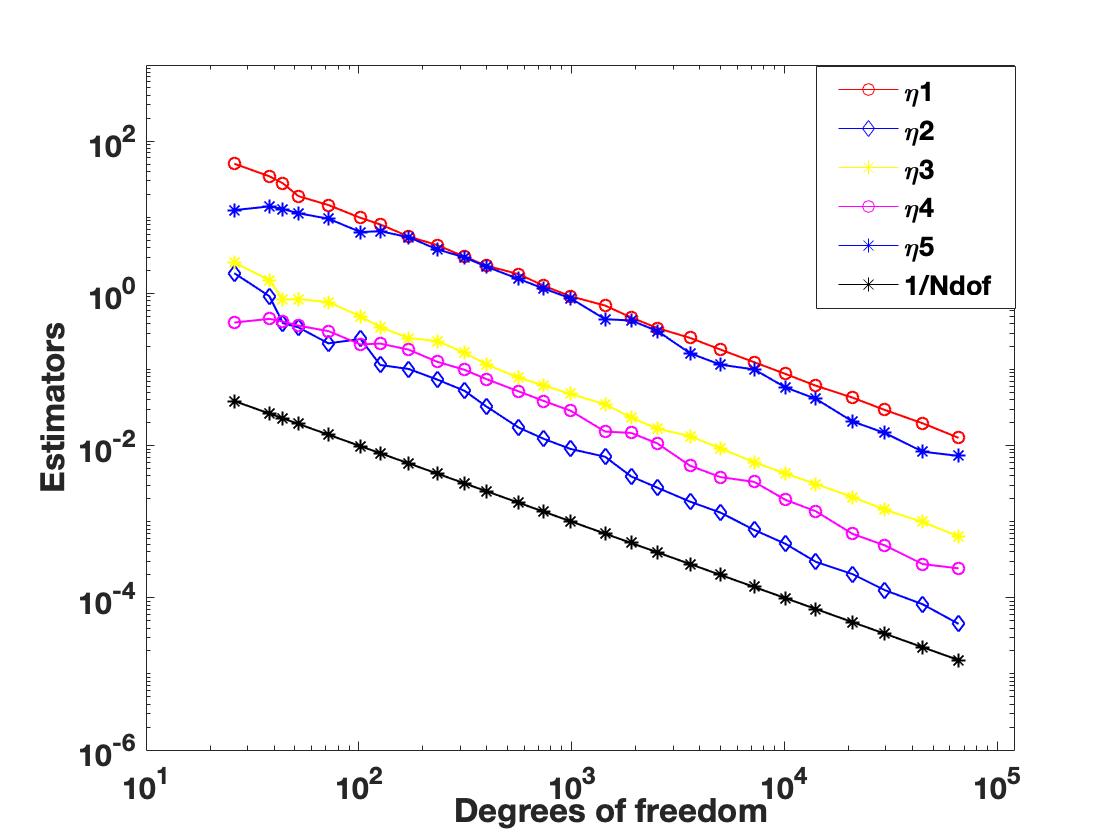}
\caption{Plot of Estimator contributions for Example 6.1.}
	\label{6}
\end{figure*}

\begin{example}
We assume the unit square = $(0,1)^2$ to be the domain $\Omega$ under consideration. The displacement fields vanishes on the top of the square i.e. zero Dirichlet boundary condition is applied on $(0,1) \times \{1\}$. The Neumann force $\b{g}$ is acting on the left and right hand side of the square namely $\{0,1\} \times (0,1)$. The bottom of the unit square is in contact with the rigid foundation and hence represent the contact boundary $\Gamma_C$. The Lame's parameters $\mu$ and $\chi$ are set to be 1.  The source term $\b{f}$ and Neumann data $\b{g}$ are computed in such a way that exact solution takes the form $\b{u}= (y^2(y-1), ~(x-2)y(1-y)e^y)$.

\par
Figure \ref{3}(a) illustrates the convergence behaviour of error and estimator with the increase in the number of degrees of freedom (Ndof). We observe that both the error and estimator converge with the optimal rate (1/Ndof), thus ensuring the reliability of the error estimator. The efficiency index depicting the efficiency of the error estimator can be seen in 
Figure \ref{3}(b).{ Figure \ref{6} ensures the convergence of each estimator contributions $\eta_i, 1\leq i \leq 5$ with the increase in degrees of freedom.  It is to be noted that the estimators $\eta_6$, $\eta_7$ vanishes for the given example as the entire contact boundary forms the active set.  } 
\end{example}
\par
 
\begin{figure*}[!ht]
	\begin{minipage}{8 cm}
	\includegraphics[width=1.1\textwidth]{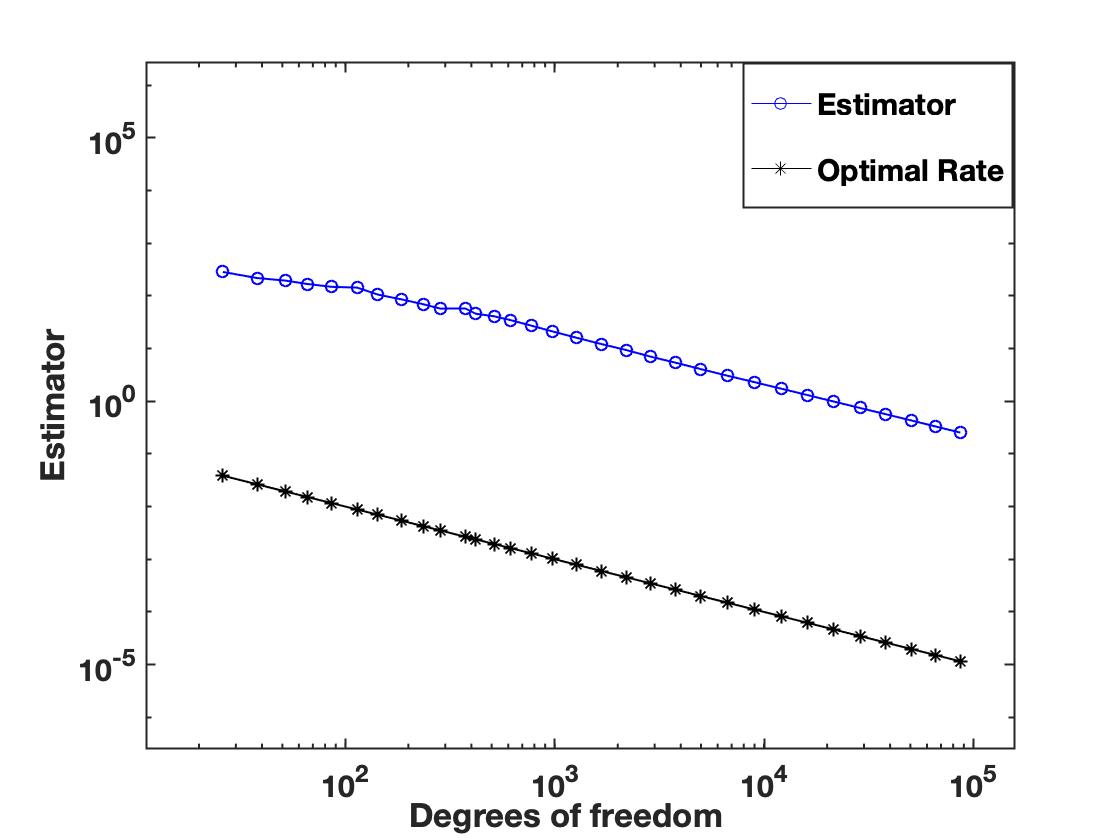}
	\end{minipage}
	\hspace{0.2cm}
	\begin{minipage}{8 cm}		
	\includegraphics[width=1.1\textwidth]{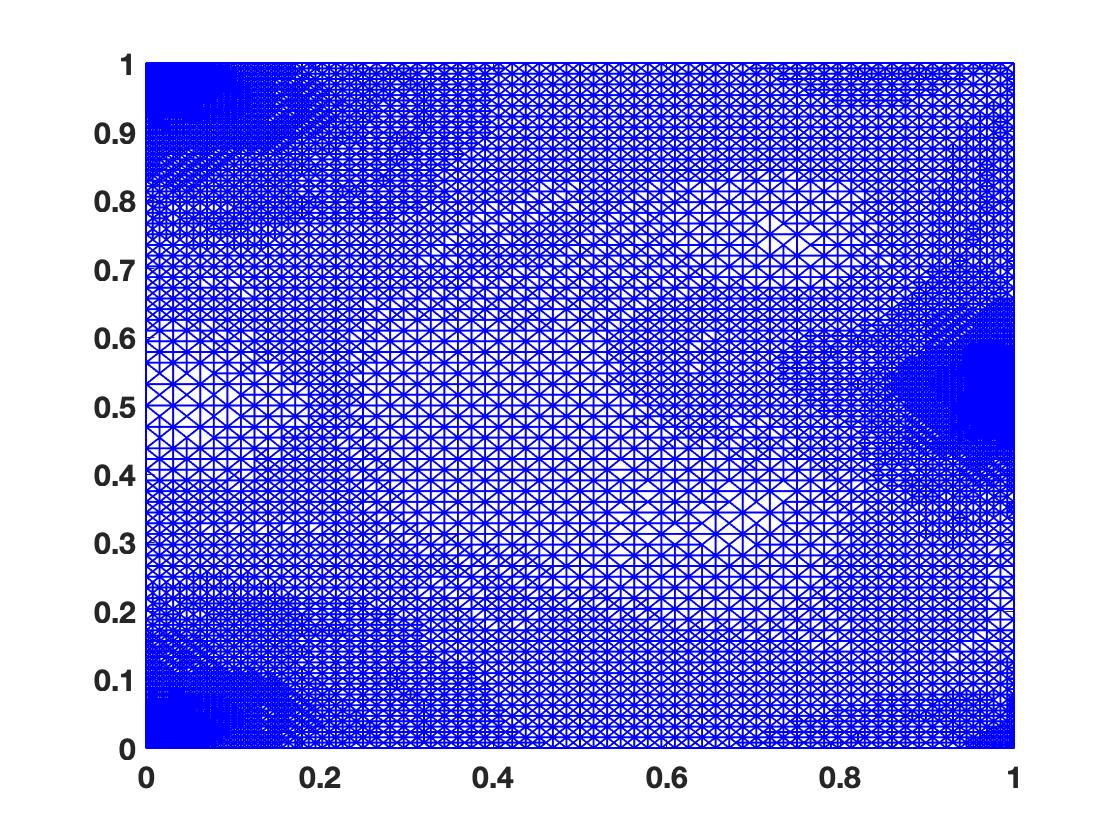}
	\end{minipage}
\caption{Estimator and Adaptive mesh for Example 6.2.}
	\label{13}
\end{figure*}
\par
\begin{figure*}[!ht]
	\includegraphics[width=0.6\textwidth]{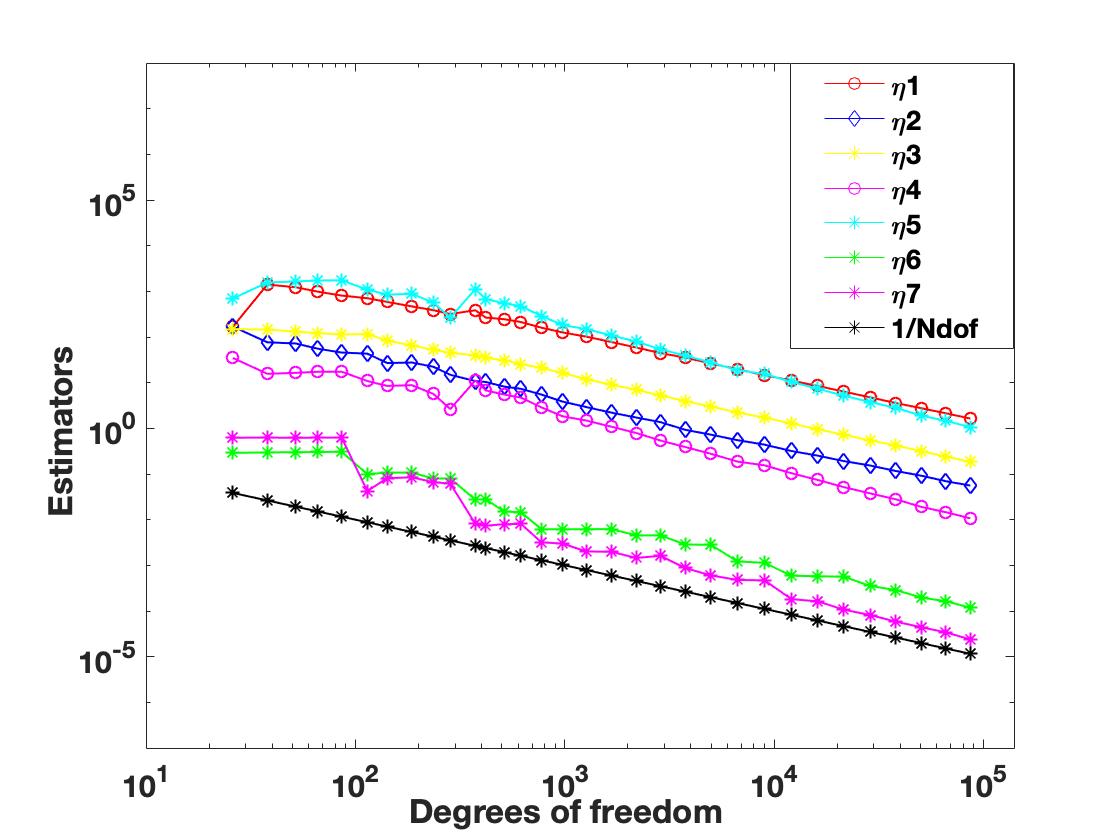}
\caption{Plot of estimator contributions for Example 6.2.}
	\label{20}
\end{figure*}
\begin{example}
In this example (motivated from \cite{Walloth:2019:Dg}), we simulate the deformation of unit elastic square which is displaced in $x$-direction towards the non zero obstacle $w(y)=-0.2+0.5|y-0.5|$. The Dirichlet boundary is set to be on the left side of elastic square at $x = 0$ with the non homogeneous condition $\b{u} = (0.1,0)$ on $\Gamma_D$ while the force density $\b{f}$ and the Neumann forces $\b{g}$ are set to zero. The Poisson ratio is $\nu$ = 0.3 and the Young’s modulus is $E=500$.   Due to non-zero obstacle, the following two terms in the error estimator will correspondingly change to $d_p= \int_{\tilde{\gamma}_{p,C}} (w-u^h_1)^+\phi_p~ds$ and 
$\eta_7 = ~\|(u^h_1-w)^{+}\|_{H^{\frac{1}{2}}(\Gamma_C)}$.  Figure $\ref{13}$(a) illustrates the convergence behaviour of estimator on the adaptive mesh as degrees of freedom increase. It is evident from the figure that the estimator converges optimally. Figure $\ref{13}$(b) depicts the adaptive mesh refinement at certain level. The high mesh refinement is observed near the intersection of disjoint Neumann and Dirichlet boundaries and also  near the free boundary region. The convergence behavior of estimator contributions is illustrated in Figure $\ref{20}$.\\
\end{example}

\textbf{Declarations.} This manuscript has no associated data.\\

\end{document}